\documentclass{amsart}
\usepackage{amsfonts}
\usepackage{amsmath}
\usepackage{amssymb}
\usepackage{amsthm}
\usepackage{booktabs}
\usepackage{mathtools, comment}
\usepackage[dvipsnames]{xcolor}
\usepackage[shortlabels]{enumitem}

\usepackage{pdfpages}

\usepackage[margin= 0.9 in]{geometry}
\usepackage{varwidth}
\usepackage{subcaption}
\usepackage{multirow}
\usepackage{rotating}
\usepackage[boxsize=1em]{ytableau}

\usepackage{bbm}

\definecolor{darkred}{rgb}{0.7,0,0}
\definecolor{sapphire}{HTML}{1B4DC2}
\definecolor{copperred}{HTML}{DA6244}
\definecolor{teal}{HTML}{247BA0}
\definecolor{plum}{HTML}{9948AB}
\definecolor{viridian}{HTML}{0AAEAB}
\definecolor{darkcyan}{HTML}{099A98}
\definecolor{forest}{HTML}{2B9B0C}
\definecolor{rose}{HTML}{C10091}
\definecolor{pumpkin}{HTML}{E47604}
\definecolor{sunflower}{HTML}{F6AE2D}
\definecolor{maize}{HTML}{FDE34F}
\definecolor{saffron}{HTML}{DAA520}
\definecolor{auburn}{HTML}{A73937}

\colorlet{leftcolor}{copperred}
\colorlet{rightcolor}{teal}

\colorlet{kstrip}{rose}
\definecolor{A1strip}{HTML}{8AB017}
\definecolor{A2strip}{HTML}{359F18}
\definecolor{A3strip}{HTML}{006611}
\definecolor{B1strip}{HTML}{3EA9B3}
\definecolor{B2strip}{HTML}{0D69B0}
\definecolor{B3strip}{HTML}{0E0999}
\definecolor{B4strip}{HTML}{640999}
\definecolor{D1strip}{HTML}{E8C01E}
\definecolor{D2strip}{HTML}{DAA11C}
\definecolor{D3strip}{HTML}{D3821F}
\definecolor{E1strip}{HTML}{C7572E}
\definecolor{E2strip}{HTML}{AB3421}
\colorlet{Bstrip}{B2strip}
\colorlet{Astrip}{A2strip}

\definecolor{ssyt1}{HTML}{BEE8D8}
\definecolor{ssyt2}{HTML}{FFF2AA}
\definecolor{ssyt3}{HTML}{C7BBFF}
\definecolor{ssyt4}{HTML}{FFBFF7}
\definecolor{bssyt1}{HTML}{1EA974}
\definecolor{bssyt2}{HTML}{D5A618}
\definecolor{bssyt3}{HTML}{4A2DCB}
\definecolor{bssyt4}{HTML}{B81FA4}

\colorlet{IminColor}{forest}
\colorlet{IpluColor}{darkcyan}
\colorlet{kColor}{plum}

\usepackage[colorlinks, allcolors=sapphire, citecolor=viridian]{hyperref}
\usepackage[noabbrev,nameinlink]{cleveref}

\usepackage{tikz}
\usetikzlibrary{calc, shapes,arrows,positioning,cd, patterns, patterns.meta}
\tikzset{>=stealth',
  head/.style = {fill = white, text=black},
  plaque/.style = {draw, rectangle, minimum size = 10mm},
  pil/.style={->,thick},
  junct/.style = {draw,circle,inner sep=0.5pt,outer sep=0pt, fill=black}
  }
\tikzstyle{opLabel}=[fill=white, inner sep=2pt, rounded corners, pos=.3]
\tikzstyle{edgeLabel}=[fill=white, inner sep=2pt, rounded corners]
\tikzstyle{v}=[draw, fill =black, circle, inner sep=0pt, minimum size=2pt] 

\tikzstyle{Bv}=[draw, fill =black, circle, inner sep=0pt, minimum size=3pt] 
\tikzstyle{xMapsto}=[line width=.6pt, |-{Classical TikZ Rightarrow[length=1mm]}]

\tikzstyle{xRightArrow}=[line width=.6pt, -{Classical TikZ Rightarrow[length=1mm]}]
\tikzstyle{intEdge}=[Bracket-Bracket, shorten <= -.5pt, shorten >= -.5pt]

\tikzstyle{TabNode}=[inner sep=2pt, fill=white, rounded corners]
\tikzstyle{LabEdge}=[inner sep=1.75pt, fill=white, rounded corners]

\tikzstyle{box relations}=[inner sep=0, darkcyan, fill = white, circle, opacity = .8, text opacity=1, font=\footnotesize]

\newcommand\TIKZ[2][]{\begin{tikzpicture}[baseline={([yshift=-.8ex]current bounding box.center)}, #1]#2\end{tikzpicture}}

\tikzset{myarrow/.tip={_[sep=-1.9pt].To[length=3pt]}}
\tikzstyle{myimplies}=[double, double distance=1.5pt, -myarrow, shorten >=2pt]

\newcounter{r}
\newcounter{s}

\newcommand\Part[1]{
        \setcounter{r}{1}
	 \foreach \x in {#1}{
 	{\ifnum\value{r}=1
		\draw (0,\value{r}-1)--(\x,\value{r}-1); 
		\fi}
	\draw (0,\value{r}) to (\x,\value{r});
   	\foreach \y in {0, ..., \x} {\draw (\y,\value{r})--(\y,\value{r}-1);}
	\addtocounter{r}{1}
 }}

\newcounter{finalrow}
\newcommand\Tableau[2][\small]{
        \setcounter{r}{0}
        \setcounter{s}{0}
        \foreach \x [count = \c from 1] in {#2} {
		\foreach \y [count = \d from 1] in \x{
			\node[font=#1] at (\d-.5,\c-.5) {$\y$};
			\draw (\d,\c) to (\d,\c-1);
			{\ifnum\d=1
				\draw (0,\c) to (0,\c-1);
				\fi}
			\setcounter{r}{\d}
		}
	\pgfmathsetmacro{\rowlength}{max(\value{r},\value{s})}
	\draw (0,\c-1) to (\rowlength,\c-1); 
		\setcounter{s}{\value{r}}
		\setcounter{finalrow}{\c}
		}
	\draw (0,\value{finalrow})--(\value{s},\value{finalrow}); 
		}

\newcommand\Comp[1]{
        \setcounter{r}{1}
        \setcounter{s}{0}
	 \foreach \x in {#1}{
	{\ifnum\value{s}<\x
		\draw (0,\value{r}-1)--(\x,\value{r}-1); 
		\else
		\draw (0,\value{r}-1)--(\value{s},\value{r}-1); 
		\fi}
   	\foreach \y in {0, ..., \x} {\draw (\y,\value{r})--(\y,\value{r}-1);}
	\addtocounter{r}{1}
        \setcounter{s}{\x}
 }
\draw (0,\value{r}-1)--(\value{s},\value{r}-1); 
 }

\def\suffScale{.35}
\def\CZero{4}
\def\COne{10}
\def\CTwo{13}
\def\CThree{16}
\def\suffR{10}
\def\suffHt{11} 
\def\suffallHt{13}
\def\abit{1.2pt}
\def\suffCT{
\node[left, black!70] at (0, \suffR-.5) {row $r_0$};
\draw[black!50, thin, fill=black!10]
(0,0) to (12,0) coordinate (top1) 
\foreach \x [count = \r from 2] in 
	{\COne+6,\COne+7,7,\CZero,\COne+1,\CZero+2,2,\CZero+4,0}{
	to ++(0,1) coordinate (top\r) 
	to (\x,\r-1) coordinate (bot\r) 
	} 
	 to (0,0);
\path (bot\suffR) to ++(0,1) coordinate (bot\suffHt);
\coordinate (B1-left) at ([shift={(.5*\abit,\abit)}]0,\suffR);
	\coordinate (B1-right) at ([shift={(-\abit,-\abit)}]\CZero-1, \suffR+1);
	\path (bot5) to ++(\abit,.5*\abit) coordinate (B2-1-left) to ++(2,1) to ++(0, -\abit) coordinate (B2-1-right);
	\path (bot7) to ++(\abit,.5*\abit) coordinate (B2-2-left) to ++(2,1) to ++(0, -\abit) coordinate (B2-2-right);
	\path (bot9) to ++(\abit,.5*\abit) coordinate (B2-3-left) to ++(2,1) to ++(-\abit, -\abit) coordinate (B2-3-right);
\coordinate (B3-left) at ([shift={(\abit,\abit)}]\COne, \suffR-1);
	\coordinate (B3-right) at ([shift={(-\abit,-\abit)}]\CTwo, \suffR);
\coordinate (B'3-left) at ([shift={(0,1)}]B3-left);
	\coordinate (B'3-right) at ([shift={(0,1)}]B3-right);
\path (bot2) to ++(\abit,.5*\abit) coordinate (B4-left) to ++(2,1) to  ++(-2*\abit,-\abit) coordinate (B4-right);
\coordinate (A1-left) at ([shift={(.5*\abit,\abit)}]0, \suffR-1);
	\coordinate (A1-right) at ([shift={(-\abit,-\abit)}]\CZero-1, \suffR);
\path (\CZero, \suffR-1) to ++(\abit,\abit) coordinate (A2-left) to ++(0, 1) coordinate (B'2-left);
	\path (\COne, \suffR) to ++(-\abit,-\abit) coordinate (A2-right) to ++(0, 1) coordinate (B'2-right); 
\path (bot6) to ++(\abit,.5*\abit) coordinate (A3-1-left) to ++(3,1) to  ++(-2*\abit,-\abit) coordinate (A3-1-right);
\path (bot3) to ++(\abit,.5*\abit) coordinate (A3-2-left) to ++(2,1) to  ++(-2*\abit,-\abit) coordinate (A3-2-right);
\path (\CZero-1,\suffR-1) to ++(\abit,\abit) coordinate (k-left) 
	 (\CZero,\suffR) to ++(-\abit,-\abit) coordinate (k-right);
\coordinate (k'-left) at ([shift={(0,1)}]k-left); \coordinate (k'-right) at ([shift={(0,1)}]k-right); } 

\def\suffYT{
\TIKZ[scale = \suffScale, font=\footnotesize]{
\suffCT
\foreach \x/\c in {\CZero/0, \COne/1, \CTwo/2}{
	\draw[black!70, densely dotted] 
		(\x-.5,\suffHt+.5) node[above, inner sep=0pt]{\begin{tabular}{c}col\\[-3pt]$c_\c$\end{tabular}} to ++(0, -1.5);
}
	\draw[A1strip, thick, fill=A1strip!20] (A1-left) rectangle node[A1strip!80!black] {$A_1$} (A1-right); 
	\draw[A2strip, thick, fill=A1strip!20] (A2-left) rectangle node[A2strip!80!black]{$A_2$} (A2-right); 
	\draw[A3strip, thick, fill=A3strip!20] (A3-1-left) rectangle node[A3strip!80!black] {$A_3$} (A3-1-right); 
	\draw[A3strip, thick, fill=A3strip!20] (A3-2-left) rectangle (A3-2-right); 
	\path (A3-1-right) to ++(0,-.5) to ++(0, .5*\abit) coordinate (A3-v1) to ++(1.5,0) coordinate (A3-v2)
		(A3-2-left) to ++(0,.5) to ++(0, -.5*\abit) coordinate (A3-v4) to ++(-1.5,0) coordinate (A3-v3);
	\draw[A3strip, thick, densely dotted, rounded corners] (A3-v1) to (A3-v2) to (A3-v3) to (A3-v4);
	\draw[B1strip!80!black, fill = B1strip!20, thick] (B1-left) rectangle node[B1strip!80!black] {$B_1$} (B1-right); 
	\foreach \x in {1,2,3}{\draw[B2strip, thick, fill=B2strip!20] (B2-\x-left) rectangle coordinate (B2-\x-label)  (B2-\x-right);}
		\node[B2strip!80!black] at (B2-2-label) {$B_2$};
	\foreach \bot [count=\top from 2] in {1,2}{\draw[B2strip, thick] (B2-\bot-right) to (B2-\top-left);}
	\draw[B3strip, fill = B3strip!20, thick] (B3-left) rectangle node[B3strip!80!black] {$B_3$} (B3-right); 
	\draw[B4strip, fill = B4strip!20, thick] (B4-left) rectangle node[B4strip!80!black] {$B_4$} (B4-right); 
	\draw[thick, kstrip, fill = kstrip!20] (k-left) rectangle node[kstrip!90!black]{$k$} (k-right);
}
}

\def\suffIYT{
\TIKZ[scale = \suffScale, font=\footnotesize]{
\suffCT
\foreach \x/\c in {\CZero/0, \COne/1, \CTwo/2}{
	\draw[black!70, densely dotted] 
		(\x-.5,\suffHt+.5) node[above, inner sep=0pt]{\begin{tabular}{c}col\\[-3pt]$c_\c$\end{tabular}} to ++(0, -.5);
}
	\draw[A1strip, thick, fill=A1strip!20] (A1-left) rectangle node[A1strip!80!black] {$A_1$} (A1-right); 
	\draw[B2strip, thick, fill=B2strip!20] (B'2-left) rectangle node[B2strip!80!black]{$B_2$} (B'2-right); 
	\draw[A3strip, thick, fill=A3strip!20] (A3-1-left) rectangle node[A3strip!80!black] {$I \cdot A_3$} (A3-1-right); 
	\draw[A3strip, thick, fill=A3strip!20] (A3-2-left) rectangle (A3-2-right); 
	\path (A3-1-right) to ++(0,-.5) to ++(0, .5*\abit) coordinate (A3-v1) to ++(1.5,0) coordinate (A3-v2)
		(A3-2-left) to ++(0,.5) to ++(0, -.5*\abit) coordinate (A3-v4) to ++(-1.5,0) coordinate (A3-v3);
	\draw[A3strip, thick, densely dotted, rounded corners] (A3-v1) to (A3-v2) to (A3-v3) to (A3-v4);
	\draw[B1strip!80!black, fill = B1strip!20, thick] (B1-left) rectangle node[B1strip!80!black] {$I \cdot B_1$} (B1-right); 
	\foreach \x in {1,2,3}{\draw[A2strip, thick, fill=A2strip!20] (B2-\x-left) rectangle coordinate (B2-\x-label)  (B2-\x-right);}
		\node[A2strip!80!black] at (B2-2-label) {$I\!\cdot\!A_2$};
	\foreach \bot [count=\top from 2] in {1,2}{\draw[A2strip, thick] (B2-\bot-right) to (B2-\top-left);}
	\draw[B3strip, fill = B3strip!20, thick] (B'3-left) rectangle node[B3strip!80!black] {$B_3$} (B'3-right); 
	\draw[B4strip, fill = B4strip!20, thick] (B4-left) rectangle node[B4strip!80!black] {$B_4$} (B4-right); 
	\draw[thick, kstrip, fill = kstrip!20] (k'-left) rectangle node[Bv, kstrip!90!black](k'-label){} (k'-right);
	\draw[<-, kstrip!90!black, rounded corners] (k'-label) to ++(.1,-1) to ++(.5,0) node[right, inner sep=1pt] {$k+m$};
}
}

\def\suffCTall{
\node[left, black!70] at (0, \suffR-.5) {row $r_0$};
\draw[black!60, thin, fill=D2strip!20!black!20]
(0,0) to (12+2,0) coordinate (all-top1) 
\foreach \x [count = \r from 2] in 
    {\COne+6,\COne+7,7+3,\CZero+4,\COne+1,\CZero+2+4,2+\CZero+9,\CZero+4+9,0,0,\COne+2, \CZero + 1, 0}{
	to ++(0,1) coordinate (all-top\r) 
	to (\x,\r-1) coordinate (all-bot\r) 
	} 
	 to (0,0);
\draw[black!60,thin, fill=black!5]
(0,0) to (12,0) coordinate (top1) 
\foreach \x [count = \r from 2] in 
	{\COne+6,\COne+7,7,\CZero,\COne+1,\CZero+2,2,\CZero+4,0}{
	to ++(0,1) coordinate (top\r) 
	to (\x,\r-1) coordinate (bot\r) 
	} 
	 to (0,0);
\path (bot\suffR) to ++(0,1) coordinate (bot\suffHt);
\node[above right, black!60] at (1,1) {\tiny$[1, i-1]$};
\coordinate (B1-left) at ([shift={(.5*\abit,\abit)}]0,\suffR);
	\coordinate (B1-right) at ([shift={(-\abit,-\abit)}]\CZero-1, \suffR+1);
	\path (bot5) to ++(\abit,.5*\abit) coordinate (B2-1-left) to ++(2,1) to ++(0, -\abit) coordinate (B2-1-right);
	\path (bot7) to ++(\abit,.5*\abit) coordinate (B2-2-left) to ++(2,1) to ++(0, -\abit) coordinate (B2-2-right);
	\path (bot9) to ++(\abit,.5*\abit) coordinate (B2-3-left) to ++(2,1) to ++(-\abit, -\abit) coordinate (B2-3-right);
\coordinate (B3-left) at ([shift={(\abit,\abit)}]\COne, \suffR-1);
	\coordinate (B3-right) at ([shift={(-\abit,-\abit)}]\CTwo, \suffR);
\coordinate (B'3-left) at ([shift={(0,1)}]B3-left);
	\coordinate (B'3-right) at ([shift={(0,1)}]B3-right);
\path (bot2) to ++(\abit,.5*\abit) coordinate (B4-left) to ++(2,1) to  ++(-2*\abit,-\abit) coordinate (B4-right);
\coordinate (A1-left) at ([shift={(.5*\abit,\abit)}]0, \suffR-1);
	\coordinate (A1-right) at ([shift={(-\abit,-\abit)}]\CZero-1, \suffR);
\path (\CZero, \suffR-1) to ++(\abit,\abit) coordinate (A2-left) to ++(0, 1) coordinate (B'2-left);
	\path (\COne, \suffR) to ++(-\abit,-\abit) coordinate (A2-right) to ++(0, 1) coordinate (B'2-right); 
\path (bot6) to ++(\abit,.5*\abit) coordinate (A3-1-left) to ++(3,1) to  ++(-2*\abit,-\abit) coordinate (A3-1-right);
\path (bot3) to ++(\abit,.5*\abit) coordinate (A3-2-left) to ++(2,1) to  ++(-2*\abit,-\abit) coordinate (A3-2-right);
\path (\CZero-1,\suffR-1) to ++(\abit,\abit) coordinate (k-left) 
	 (\CZero,\suffR) to ++(-\abit,-\abit) coordinate (k-right);
\coordinate (k'-left) at ([shift={(0,1)}]k-left); \coordinate (k'-right) at ([shift={(0,1)}]k-right); 
\path (\CZero-1, \suffR) to ++(\abit,\abit) coordinate (D1-left) to ++(0, -1) coordinate (D1'-left);
\path (\CTwo, \suffR+1) to ++(-\abit,-\abit) coordinate (D1-right) to ++(0, -1) coordinate (D1'-right); 
\path (\CTwo, \suffR) to ++(\abit,\abit) coordinate (D2-left) to ++(0, -1) coordinate (E1-left);
\path (\CThree, \suffR+1) to ++(-\abit,-\abit) coordinate (D2-right) to ++(0, -1) coordinate (E1-right); 
\path (\CThree, \suffR) to ++(\abit,\abit) coordinate (D3-SW) to ++(0, -1) coordinate (E2-SW);
\path (\CThree+3, \suffR+1) to ++(-\abit,-\abit) coordinate (D3-NE) to ++(0, -1) coordinate (E2-NE);
\path (D3-SW) to ++(0,1) to ++(0,-2*\abit) coordinate (D3-NW);
\path (D3-NE) to ++(0,-1) to ++(0,2*\abit) coordinate (D3-SE);
\path (E2-SW) to ++(0,1) to ++(0,-2*\abit) coordinate (E2-NW);
\path (E2-NE) to ++(0,-1) to ++(0,2*\abit) coordinate (E2-SE);
}

\def\semistableCore{
\node[left, black!70] at (0, \suffR-.5) {row $r_0$};
\draw[black!60, thin, fill=D2strip!20!black!20]
(0,0) to (12+2,0) coordinate (all-top1) 
\foreach \x [count = \r from 2] in 
    {\COne+6,\COne+7,7+3,\CZero+4,\COne+1,\CZero+2+4,2+\CZero+9,\CZero+4+9,0,0,\COne+2, \CZero + 1, 0}{
	to ++(0,1) coordinate (all-top\r) 
	to (\x,\r-1) coordinate (all-bot\r) 
	} 
	 to (0,0);
\draw[black!60,thin, fill=black!5]
(0,0) to (12,0) coordinate (top1) 
\foreach \x [count = \r from 2] in 
	{\COne+6,\COne+7,7,\CZero,\COne+1,\CZero+2,2,\CZero+4,0}{
	to ++(0,1) coordinate (top\r) 
	to (\x,\r-1) coordinate (bot\r) 
	} 
	 to (0,0);
\path (bot\suffR) to ++(0,1) coordinate (bot\suffHt);
\node[above right, black!60] at (1,1) {\tiny$[1, i-1]$};
\coordinate (B1-left) at ([shift={(.5*\abit,\abit)}]0,\suffR);
	\coordinate (B1-right) at ([shift={(-\abit,-\abit)}]\CZero-1, \suffR+1);
	\path (bot5) to ++(\abit,.5*\abit) coordinate (B2-1-left) to ++(2,1) to ++(0, -\abit) coordinate (B2-1-right);
	\path (bot7) to ++(\abit,.5*\abit) coordinate (B2-2-left) to ++(2,1) to ++(0, -\abit) coordinate (B2-2-right);
	\path (bot9) to ++(\abit,.5*\abit) coordinate (B2-3-left) to ++(2,1) to ++(-\abit, -\abit) coordinate (B2-3-right);
\coordinate (B3-left) at ([shift={(\abit,\abit)}]\COne, \suffR-1);
	\coordinate (B3-right) at ([shift={(-\abit,-\abit)}]\CTwo, \suffR);
\coordinate (B'3-left) at ([shift={(0,1)}]B3-left);
	\coordinate (B'3-right) at ([shift={(0,1)}]B3-right);
\path (bot2) to ++(\abit,.5*\abit) coordinate (B4-left) to ++(2,1) to  ++(-2*\abit,-\abit) coordinate (B4-right);
\coordinate (A1-left) at ([shift={(.5*\abit,\abit)}]0, \suffR-1);
	\coordinate (A1-right) at ([shift={(-\abit,-\abit)}]\CZero-1, \suffR);
\path (\CZero, \suffR-1) to ++(\abit,\abit) coordinate (A2-left) to ++(0, 1) coordinate (B'2-left);
	\path (\COne, \suffR) to ++(-\abit,-\abit) coordinate (A2-right) to ++(0, 1) coordinate (B'2-right); 
\path (bot6) to ++(\abit,.5*\abit) coordinate (A3-1-left) to ++(3,1) to  ++(-2*\abit,-\abit) coordinate (A3-1-right);
\path (bot3) to ++(\abit,.5*\abit) coordinate (A3-2-left) to ++(2,1) to  ++(-2*\abit,-\abit) coordinate (A3-2-right);
\path (\CZero-1,\suffR-1) to ++(\abit,\abit) coordinate (k-left) 
	 (\CZero,\suffR) to ++(-\abit,-\abit) coordinate (k-right);
\coordinate (k'-left) at ([shift={(0,1)}]k-left); \coordinate (k'-right) at ([shift={(0,1)}]k-right); 
\path (\CZero-1, \suffR) to ++(\abit,\abit) coordinate (D1-SW) to ++(0, -1) coordinate (D1'-SW);
\path (\CTwo, \suffR+1) to ++(-\abit,-\abit) coordinate (D1-NE) to ++(0, -1) coordinate (D1'-NE); 
\path (\CTwo, \suffR) to ++(\abit,\abit) coordinate (D2-SW) to ++(0, -1) coordinate (E1-SW);
\path (\CThree, \suffR+1) to ++(-\abit,-\abit) coordinate (D2-NE) to ++(0, -1) coordinate (E1-NE); 
\path (\CThree, \suffR) to ++(\abit,\abit) coordinate (D3-SW) to ++(0, -1) coordinate (E2-SW);
\path (\CThree+3, \suffR+1) to ++(-\abit,-\abit) coordinate (D3-NE) to ++(0, -1) coordinate (E2-NE);
\path (D3-SW) to ++(0,1) to ++(0,-2*\abit) coordinate (D3-NW);
\path (D3-NE) to ++(0,-1) to ++(0,2*\abit) coordinate (D3-SE);
\path (E2-SW) to ++(0,1) to ++(0,-2*\abit) coordinate (E2-NW);
\path (E2-NE) to ++(0,-1) to ++(0,2*\abit) coordinate (E2-SE);
\path (D1-SW) to ++(0,1) to ++(0,-2*\abit) coordinate (D1-NW);
\path (D1-NE) to ++(0,-1) to ++(0,2*\abit) coordinate (D1-SE);
\path (E1-SW) to ++(0,1) to ++(0,-2*\abit) coordinate (E1-NW);
\path (E1-NE) to ++(0,-1) to ++(0,2*\abit) coordinate (E1-SE);
\path (D2-SW) to ++(0,1) to ++(0,-2*\abit) coordinate (D2-NW);
\path (D2-NE) to ++(0,-1) to ++(0,2*\abit) coordinate (D2-SE);
}

\def\semistableYT{
\TIKZ[scale = \suffScale, font=\footnotesize]{
\semistableCore
 \foreach \x/\lab in {A1/A_1, B1/B_1,A2/A_2,B3/B_3,B2-1/{},B2-2/B_2,B2-3/{},A3-1/A_3,A3-2/{},B4/B_4}{
 	\draw[Bstrip!50, fill = Bstrip!10, thick] (\x-left) rectangle node{$\lab$}(\x-right);}
	\path (A3-1-right) to ++(0,-.5) to ++(0, .5*\abit) coordinate (A3-v1) to ++(1.5,0) coordinate (A3-v2)
		(A3-2-left) to ++(0,.5) to ++(0, -.5*\abit) coordinate (A3-v4) to ++(-1.5,0) coordinate (A3-v3);
	\draw[Bstrip!50, thick, densely dotted, rounded corners] (A3-v1) to (A3-v2) to (A3-v3) to (A3-v4);
	\foreach \bot [count=\top from 2] in {1,2}{\draw[Bstrip!50, thick] (B2-\bot-right) to (B2-\top-left);}
 	\draw[thick, Bstrip, fill = Bstrip!30] (k-left) rectangle node[Bstrip!80!black]{$k$} (k-right);
	\draw[D1strip, fill = D1strip!20, thick] (D3-NE) to (D1-NW) to (D1-SW) to (D3-SE);
	\draw[E1strip, fill = E1strip!20, thick] (E2-NE) to (E1-NW) to (E1-SW) to (E2-SE);
	\path  (D3-NE) to node[D2strip!80!black] {$\dots$} (D3-SE);
	\path  (E2-NE) to node[E1strip!80!black] {$\dots$} (E2-SE);
}
}

\def\semistableIYT{
\TIKZ[scale = \suffScale, font=\footnotesize]{
\semistableCore
 \foreach \x/\lab in {A1/A_1, B1/I\cdot B_1,B'2/B_2,B'3/B_3,B2-1/{},B2-2/{I\! \cdot \! A_2},B2-3/{},A3-1/I \cdot A_3,A3-2/{},B4/B_4}{
 	\draw[Bstrip!50, fill = Bstrip!10, thick] (\x-left) rectangle node{$\lab$}(\x-right);}
	\path (A3-1-right) to ++(0,-.5) to ++(0, .5*\abit) coordinate (A3-v1) to ++(1.5,0) coordinate (A3-v2)
		(A3-2-left) to ++(0,.5) to ++(0, -.5*\abit) coordinate (A3-v4) to ++(-1.5,0) coordinate (A3-v3);
	\draw[Bstrip!50, thick, densely dotted, rounded corners] (A3-v1) to (A3-v2) to (A3-v3) to (A3-v4);
	\foreach \bot [count=\top from 2] in {1,2}{\draw[Bstrip!50, thick] (B2-\bot-right) to (B2-\top-left);}
	\draw[thick, Bstrip, fill = Bstrip!30] (k'-left) rectangle node[Bstrip!80!black]{$k'$} (k'-right);
	\draw[D1strip, fill = D1strip!20, thick] (D1'-SW) rectangle (D1'-NE); 
	\fill[E2strip!20] 
		(D3-NE) to (D2-NW) to (E1-SW) to (E2-SE);
	\draw[pattern={Lines[angle=0, distance=4pt, line width=2pt]}, pattern color=D2strip!20, draw=D3strip, thick] 
		(D3-NE) to (D2-NW) to (E1-SW) to (E2-SE);
	\path  (D3-NE) to node[D3strip!80!black] {$\dots$} (D3-SE);
	\path  (E2-NE) to node[D3strip!80!black] {$\dots$} (E2-SE);
	\draw[D3strip, thick] ([shift={(0, 2*\abit)}]E1-NW) to ([shift={(0, 2*\abit)}]E2-NE);
}
}

\def\suffYTallColor{
\TIKZ[scale = \suffScale, font=\footnotesize]{
\suffCTall
\foreach \x/\c in {\CZero/0, \COne/1, \CTwo/2, \CThree/3}{
	\draw[black!70, densely dotted] 
		(\x-.5,\suffallHt+.5) node[above, inner sep=0pt]{\begin{tabular}{c}col\\[-3pt]$c_\c$\end{tabular}} to (\x-.5, \suffHt-1);}
	\draw[A1strip, thick, fill=A1strip!20] (A1-left) rectangle node[A1strip!80!black] {$A_1$} (A1-right); 
	\draw[A2strip, thick, fill=A1strip!20] (A2-left) rectangle node[A2strip!80!black]{$A_2$} (A2-right); 
	\draw[A3strip, thick, fill=A3strip!20] (A3-1-left) rectangle node[A3strip!80!black] {$A_3$} (A3-1-right); 
	\draw[A3strip, thick, fill=A3strip!20] (A3-2-left) rectangle (A3-2-right); 
	\path (A3-1-right) to ++(0,-.5) to ++(0, .5*\abit) coordinate (A3-v1) to ++(1.5,0) coordinate (A3-v2)
		(A3-2-left) to ++(0,.5) to ++(0, -.5*\abit) coordinate (A3-v4) to ++(-1.5,0) coordinate (A3-v3);
	\draw[A3strip, thick, densely dotted, rounded corners] (A3-v1) to (A3-v2) to (A3-v3) to (A3-v4);
	\draw[B1strip!80!black, fill = B1strip!20, thick] (B1-left) rectangle node[B1strip!80!black] {$B_1$} (B1-right); 
	\foreach \x in {1,2,3}{\draw[B2strip, thick, fill=B2strip!20] (B2-\x-left) rectangle coordinate (B2-\x-label)  (B2-\x-right);}
		\node[B2strip!80!black] at (B2-2-label) {$B_2$};
	\foreach \bot [count=\top from 2] in {1,2}{\draw[B2strip, thick] (B2-\bot-right) to (B2-\top-left);}
	\draw[B3strip, fill = B3strip!20, thick] (B3-left) rectangle node[B3strip!80!black] {$B_3$} (B3-right); 
	\draw[B4strip, fill = B4strip!20, thick] (B4-left) rectangle node[B4strip!80!black] {$B_4$} (B4-right); 
	\draw[thick, kstrip, fill = kstrip!20] (k-left) rectangle node[kstrip!90!black]{$k$} (k-right);
	\draw[D1strip, fill = D1strip!20, thick] (D1-left) rectangle node[D1strip!80!black] {$D_1$} (D1-right); 
	\draw[D2strip, fill = D2strip!20, thick] (D2-left) rectangle node[D2strip!80!black] {$D_2$} (D2-right); 
	\draw[E1strip, fill = E1strip!20, thick] (E1-left) rectangle node[E1strip!80!black] {$E_1$} (E1-right); 
	\draw[D3strip, fill = D3strip!20, thick] (D3-NE) to (D3-NW) to (D3-SW) to (D3-SE);
	\path  (D3-NE) to node[D3strip!80!black] {$D_3$} (D3-SW);
	\path  (D3-NE) to node[E2strip!80!black] {$\dots$} (D3-SE);
	\draw[E2strip, fill = E2strip!20, thick] (E2-NE) to (E2-NW) to (E2-SW) to (E2-SE);
	\path  (E2-NE) to node[E2strip!80!black] {$E_2$} (E2-SW);
	\path  (E2-NE) to node[E2strip!80!black] {$\dots$} (E2-SE);
	\path  (E1-right) to node[rotate = 90] {$<$} ++(-1,0);
	\path  (E2-NW) to node[rotate = 90] {$\geqslant$} ++(1,0);
}
}

\def\suffIYTallColor{
\TIKZ[scale = \suffScale, font=\footnotesize]{
\suffCTall
\foreach \x/\c in {\CZero/0, \COne/1, \CTwo/2, \CThree/3}{
	\draw[black!70, densely dotted] 
		(\x-.5,\suffallHt+.5) node[above, inner sep=0pt]{\begin{tabular}{c}col\\[-3pt]$c_\c$\end{tabular}} to (\x-.5, \suffHt);}
	\draw[A1strip, thick, fill=A1strip!20] (A1-left) rectangle node[A1strip!80!black] {$A_1$} (A1-right); 
	\draw[B2strip, thick, fill=B2strip!20] (B'2-left) rectangle node[B2strip!80!black]{$B_2$} (B'2-right); 
	\draw[A3strip, thick, fill=A3strip!20] (A3-1-left) rectangle node[A3strip!80!black] {$I \cdot A_3$} (A3-1-right); 
	\draw[A3strip, thick, fill=A3strip!20] (A3-2-left) rectangle (A3-2-right); 
	\path (A3-1-right) to ++(0,-.5) to ++(0, .5*\abit) coordinate (A3-v1) to ++(1.5,0) coordinate (A3-v2)
		(A3-2-left) to ++(0,.5) to ++(0, -.5*\abit) coordinate (A3-v4) to ++(-1.5,0) coordinate (A3-v3);
	\draw[A3strip, thick, densely dotted, rounded corners] (A3-v1) to (A3-v2) to (A3-v3) to (A3-v4);
	\draw[B1strip!80!black, fill = B1strip!20, thick] (B1-left) rectangle node[B1strip!80!black] {$I \cdot B_1$} (B1-right); 
	\foreach \x in {1,2,3}{\draw[A2strip, thick, fill=A2strip!20] (B2-\x-left) rectangle coordinate (B2-\x-label)  (B2-\x-right);}
		\node[A2strip!80!black] at (B2-2-label) {$I\!\cdot\!A_2$};
	\foreach \bot [count=\top from 2] in {1,2}{\draw[A2strip, thick] (B2-\bot-right) to (B2-\top-left);}
	\draw[B3strip, fill = B3strip!20, thick] (B'3-left) rectangle node[B3strip!80!black] {$B_3$} (B'3-right); 
	\draw[B4strip, fill = B4strip!20, thick] (B4-left) rectangle node[B4strip!80!black] {$B_4$} (B4-right); 
	\draw[thick, kstrip, fill = kstrip!20] (k'-left) rectangle node[kstrip!90!black]{$k'$} (k'-right);
	\draw[D1strip, fill = D1strip!20, thick] (D1'-left) rectangle node[D1strip!80!black] {$D_1$} (D1'-right); 
	\draw[D2strip, fill = D2strip!20, thick] (E1-left) rectangle node[D2strip!80!black] {$D_2$} (E1-right); 
	\draw[E1strip, fill = E1strip!20, thick] (D2-left) rectangle node[E1strip!80!black] {$E_1$} (D2-right); 
	\draw[D3strip, fill = D3strip!20, thick] (D3-NE) to (D3-NW) to (D3-SW) to (D3-SE);
	\path  (D3-NE) to node[D3strip!80!black] {$D_3$} (D3-SW);
	\path  (D3-NE) to node[E2strip!80!black] {$\dots$} (D3-SE);
	\draw[E2strip, fill = E2strip!20, thick] (E2-NE) to (E2-NW) to (E2-SW) to (E2-SE);
	\path  (E2-NE) to node[E2strip!80!black] {$E_2$} (E2-SW);
	\path  (E2-NE) to node[E2strip!80!black] {$\dots$} (E2-SE);
	\path  (E1-right) to node[rotate = 90] {$>$} ++(-1,0);
	\path  (E2-NW) to node[rotate = 90] {$\geqslant$} ++(1,0);
}
}

\def\ABtwoCore{
\coordinate (T) at (5.5, 12.5);
\filldraw[black!20, pattern=north east lines, pattern color=black!20] 
	(0,9.5) coordinate (row r) rectangle ++(11,-.5);
\filldraw[black!10]
(0,0) to (1,0) coordinate (x1)
\foreach \x [count = \r from 2] in 
	{5, 7, 10, 5, 5, 10,11,8,0}{
	to ++(0,1) coordinate (y\r) 
	to (\x,\r-1) coordinate (x\r) 
	} 
	to ++(0,1);
\draw[black!40] (x1) 
\foreach \r in {2, 3, ..., 8}{to (y\r) to (x\r)};
\draw [densely dotted] 
	(x1) to (0,0) to (x10) to ++(0,1.5)
	(x8) to (y9)
	(1,0) to ++(0,9.5);
\draw[|-|, black!50] (-.5,0) to node[above, sloped]{rows $r_1$ through $r_{c_1-c_0}$} (-.5,9);}

\def\rowR{2.25}
\newcommand{\Triple}[3]
{\draw (0,\rowR) rectangle +(2,1);
		\draw (1,\rowR) to +(0,1);
		\node (x) at (0.5, \rowR+.5) {$#1$\strut};
		\node (y) at (1.5, \rowR+.5) {$#2$\strut}; 
	\draw (1,0) rectangle +(1,1);
    		\node (z) at (1.5, .5) {$#3$\strut};  
	\draw (1.5, 1) to 
		node[fill=white, sloped, inner sep = .5pt]{\footnotesize$<$} 
		+(-1, \rowR-1);
	\draw[densely dashed] (1.5, 1) to +(0, \rowR-1);
\coordinate (label) at (0,\rowR/2+.5);
\coordinate (top row) at  (2,\rowR+.5);
\coordinate (bot row) at  (2,.5);}

\newcommand{\cell}{\mathsf{cell}}

\newcommand{\col}{\mathsf{col}}
\newcommand{\Col}{\mathsf{Col}}
\newcommand{\CS}{\mathsf{CS}}
\newcommand{\cycle}{\mathsf{cycle}}

\newcommand{\QCS}{\mathsf{QCS}}
\newcommand{\QS}{\mathcal{QS}}
\newcommand{\QSym}{\mathsf{QSym}}

\newcommand{\row}{{\mathsf{row}}}
\newcommand{\shape}{\mathsf{shape}}

\newcommand{\SYT}{\mathsf{SYT}}

\newcommand{\wt}{{\mathsf{wt}}}
\newcommand{\yc}[1]{\widehat{#1}}

\newcommand{\YQS}{\mathcal{YQS}}

\newcommand{\sm}{\scalebox{.75}[1.0]{\hspace{.2pt}$-$}}

\newcommand{\fsl}{\mathfrak{sl}}

\newcommand{\xdownarrow}[1]{%
  {\left\downarrow\vbox to #1{}\right.\kern-\nulldelimiterspace}
}

\usepackage[colorinlistoftodos]{todonotes}

\newtheorem{theorem}{Theorem}[section]

\newtheorem{corollary}[theorem]{Corollary}
\newtheorem{lemma}[theorem]{Lemma}

\newtheorem{proposition}[theorem]{Proposition}

\theoremstyle{definition}
\newtheorem{definition}[theorem]{Definition}

\newtheorem{remark}[theorem]{Remark}
\newtheorem{example}[theorem]{Example}

\numberwithin{equation}{section}

\newcommand{\defncolor}{\color{sapphire!70!black}}
\newcommand{\defn}[1]{{\defncolor\emph{#1}}} 

\title[Contractions and applications of crystal skeletons]
{Contractions and applications of crystal skeletons: \\Young quasisymmetric and Stanley symmetric functions}

\author[Brauner]{Sarah Brauner}
\address[S.\ Brauner]{Department of Mathematics, University of Pennsylvania, Philadelphia, PA, USA}
\email{sarahbrauner@gmail.com}
\urladdr{\href{https://www.sarahbrauner.com/}{https://www.sarahbrauner.com/}}

\author[Daugherty]{Zajj Daugherty}
\address[Z.\ Daugherty]{Department of Mathematics and Statistics, Reed College, 3203 SE Woodstock Blvd, Portland, OR 97202-8199, USA}
\email{zdaugherty@reed.edu}
\urladdr{\url{https://people.reed.edu/~zdaugherty/}}

\author[Mason]{Sarah Mason}
\address[S.\ Mason]{Wake Forest University, 
1834 Wake Forest Road, Winston-Salem, NC, USA}
\email{sarahkmason@gmail.com}
\urladdr{\href{https://sites.google.com/wfu.edu/sarahmason/}{https://sites.google.com/wfu.edu/sarahmason/}}

\author[Schilling]{Anne Schilling}
\address[A. Schilling]{Department of Mathematics, University of California, One Shields
Avenue, Davis, CA 95616-8633, U.S.A.}
\email{aschilling@ucdavis.edu}
\urladdr{\href{http://www.math.ucdavis.edu/~anne}{http://www.math.ucdavis.edu/~anne}}

\date{\today}

\begin{document}

\maketitle

\begin{abstract}
The character of a connected $\fsl_n$-crystal is a Schur polynomial; the crystal can be further decomposed into quasicrystals, 
whose characters are the Gessel quasisymmetric functions. Crystal skeletons are obtained by contracting quasicrystals within 
crystal graphs. They generalize dual equivalence graphs, and can be used to prove the Schur expansion of a symmetric function
when the quasisymmetric expansion is known. In this paper, we show that the crystal skeleton can be tiled further into components
which we call quasicrystal skeletons, whose characters are Young quasisymmetric Schur functions. We characterize which edges in the crystal
skeleton move between quasicrystal skeleton components. Contracting the quasicrystal skeleton components yields Bruhat order.
We illustrate how these tools can be applied to symmetric functions by analyzing the Stanley symmetric functions.
\end{abstract}

\tableofcontents

\section{Introduction}
In this paper, we study various contractions of \defn{$\fsl_n$-crystals} and their connections to symmetric and quasisymmetric functions, as summarized 
in Figure~\ref{figure.summary}.

\defn{Crystal bases} provide combinatorial tools to study the representation theory of Lie algebras (see~\cite{BumpSchilling.2017} for details). 
Characters of $\mathfrak{sl}_n$-crystals are \defn{Schur polynomials}. In the same way that a Schur function can be written as a sum of Gessel's 
fundamental quasisymmetric functions $F_\alpha$ ~\cite{Gessel.1984}, crystals can be decomposed into \defn{quasicrystals} whose characters are 
the $F_\alpha$. Quasicrystals were studied in~\cite{CMRR.2023,CMRR.2025,MG.2023}.

Quasicrystals are well-understood, but how they tile the crystal is far more complex. In an effort to understand this tiling,
Maas-Gari\'epy~\cite{MG.2023} introduced \defn{crystal skeletons} by contracting quasicrystals within crystal graphs.
One can thus understand the crystal skeleton as the representation theoretic manifestation of Gessel's formula 
\begin{equation}
\label{equation.s F}
    s_\lambda = \sum_{T \in \mathsf{SYT}(\lambda)} F_{\mathsf{Des}(T)},    
\end{equation}
where $\mathsf{SYT}(\lambda)$ is the set of standard Young tableaux of shape $\lambda$ and $\mathsf{Des}(T)$ is the descent composition of 
$T\in \mathsf{SYT}(\lambda)$.

In~\cite{BCDS.2025}, the authors give combinatorial and axiomatic descriptions of crystals skeletons, proving that crystal skeletons generalize dual 
equivalence graphs, and have many intriguing properties such as branching rules, a Lusztig involution symmetry, and a subcrystal property. The axiomatic 
framework in \cite{BCDS.2025} suggests that crystal skeletons can be used as a means to prove Schur expansions of symmetric functions
when the quasisymmetric expansion into the fundamental basis is known. 

In this paper, we show that the crystal skeleton can be further tiled  into components which we 
call \defn{quasicrystal skeletons}, whose characters are the \defn{Young quasisymmetric Schur functions}
$\YQS_\alpha$, indexed by compositions $\alpha$~\cite{HLMvW.2011,LMvW.2013}.
 The tiling of the crystal skeleton into quasicrystal skeletons is the representation theoretic interpretation of the formula
\begin{equation}
\label{equation.s YQS}
	s_{\lambda} = \sum_{\alpha} \YQS_{\alpha},
\end{equation}
where the sum is over all compositions $\alpha$ that rearrange $\lambda$. Contracting the quasicrystal skeleton components---analogous to contracting
the quasicrystals inside the crystal to obtain the crystal skeleton---yields \defn{Bruhat order}. 
This is summarized in Figure~\ref{figure.summary}.

\begin{figure}[t]
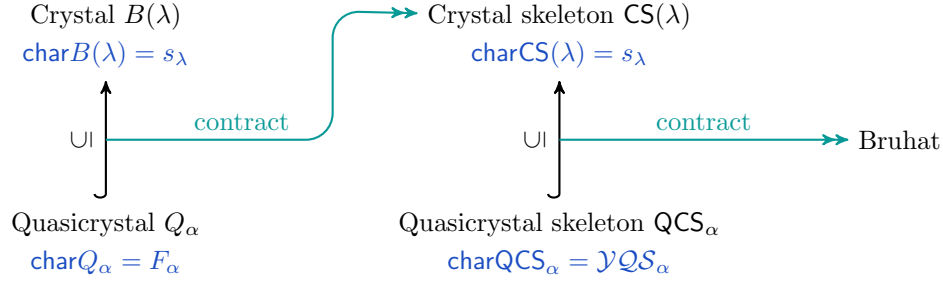

\centering
\TIKZ[thick]{
\begin{scope}
\node[label={[name=cB]below, yshift=5pt, sapphire:
	{$\mathsf{char} B(\lambda) = s_\lambda$\strut}}]
	(B) at (0, 2.75) {Crystal $B(\lambda)$\strut};
\node[label={below, yshift=5pt, sapphire}:
	{$\mathsf{char} Q_\alpha = F_\alpha$\strut}]
	(Q) at (0, 0) {Quasicrystal $Q_\alpha$\strut};
\draw[{Hooks[right, scale width=1.5]}->] (Q) to 
	node[above, rotate=90, outer sep=2pt]{$\subseteq$}
	coordinate (crystal ref) (cB);
\node[label={[name=cCS]below, yshift=5pt, sapphire:
	{$\mathsf{char} \mathsf{CS}(\lambda) = s_\lambda$\strut}}]
	(CS) at (6,2.75) {Crystal skeleton $\mathsf{CS}(\lambda)$\strut};
\node[label={below, yshift=5pt, sapphire}:
	{$\mathsf{char} \mathsf{QCS}_\alpha = \YQS_\alpha$\strut}]
	(QCS) at (6,0) {Quasicrystal skeleton $\mathsf{QCS}_\alpha$\strut};
\draw[{Hooks[right, scale width=1.5]}->] (QCS) to 
	node[above, rotate=90, outer sep=2pt]{$\subseteq$}
	coordinate (skeleton ref) (cCS);
\end{scope}
\path (skeleton ref) to ++(4.5,0) coordinate (coord_Bruhat);
\node (Bruhat) at (coord_Bruhat) {Bruhat};
\begin{scope}[darkcyan]
\draw[-{>>}, rounded corners=10] (crystal ref) to node[above, sloped, pos=.6]{contract} 
	++(3,0) to ++(0,1.7) to (CS);
\draw[-{>>}] (skeleton ref) to node[above, sloped]{contract} (Bruhat);
\end{scope}
}

\caption{Illustration of the relationships between the graphs and characters studied in this paper. 
See \eqref{equation.s F} and \eqref{equation.s YQS} for corresponding character transitions.
\label{figure.summary}}
\end{figure}

Passing from the Young quasisymmetric Schur expansion to the Schur expansion is straightforward; thus understanding the quasicrystal skeleton 
decomposition of the crystal skeleton is crucial in finding Schur expansions. In this paper, we characterize which edges in the crystal skeleton move 
between quasicrystal skeleton components (see Theorem~\ref{theorem.component change}).

We then illustrate how these tools can be applied to symmetric functions by analyzing the Stanley symmetric functions for any permutation.
Related work on Stanley symmetric functions indexed by long permutations comes from the study of tubing lattices by Dahlberg and
Fishel~\cite{DF.2024}. It uses the projection from the weak order to the tubing lattice by Barnard and McConville~\cite{BM.2021},
where the tubing lattice is an orientation of the 1-skeleton of a graph associahedron which is a specialization of the 
graph permutahedron of Postnikov~\cite{Postnikov.2009}. Our approach using the crystal skeleton works for any permutation, not just the 
long permutation.

In~\cite{BGNST.2025}, a geometric interpretation of quasisymmetric functions in terms of quasisymmetric flag varieties is given.
Our work can be viewed as a complementary approach that provides a crystal-theoretic interpretation of quasisymmetric functions.

The paper is organized as follows. In \S\ref{section.CS}, we give the background on crystal skeletons. In \S\ref{section.QS}, we review
the combinatorics of Young quasisymmetric Schur functions. In \S\ref{section.QCS}, we present our main results: we introduce quasicrystal skeletons
and characterize crystal skeleton edges that move between components (see Theorem~\ref{theorem.component change}). We also study the connectivity 
of the quasicrystal skeleton components (see Theorem~\ref{theorem.connectivity}). In \S\ref{section.stanley}, we apply our new methods to the 
combinatorics of Stanley symmetric functions.

\subsection*{Acknowledgements}
This material is based on work supported by the National Science Foundation under Grant No. DMS-1929284, while the authors
were in residence at the ICERM semester program ``Categorification and Computation in Algebraic Combinatorics'' in Fall 2025.
We thank ICERM for the stimulating research atmosphere.

We also thank Susanna Fishel, \'{A}lvaro Guti\'{e}rrez, Zach Hamaker, Hsin-Chieh Liao, Olya Mandelshtam, Joseph Pappe, Benjamin Young, 
Tianyi Yu, and Mike Zabrocki for fruitful discussions. 
Further, S.\ Brauner was partially supported by NSF MSPRF DMS-2303060; S.\ Mason was partially supported by AMS-Simons PUI grant; 
and A.\ Schilling was partially supported by  Simons Foundation grant MPS-TSM-00007191.
 
\section{Crystal skeletons}
\label{section.CS}

We first review some properties of crystals, quasicrystals, and crystal skeletons.

\subsection{Crystal}
\label{ss.crystal}
Let $\lambda$ be a partition with at most $n$ parts. An \defn{$\mathfrak{sl}_n$-crystal} $B(\lambda)$ is a graph with vertices labeled by 
semistandard Young tableaux $b \in \mathsf{SSYT}_n(\lambda)$ of shape $\lambda$ over the alphabet $\{1,2,\ldots,n\}$, and edges as follows. 

Let $w=\mathsf{row}(b)$ be the row reading word of $b$.
The \defn{crystal lowering} operator $f_i$ acts on the subword of $w$ containing only the letters $i$ and $i+1$.  Consider each $i+1$ to be a left 
parenthesis and each $i$ to be a right parenthesis and pair them as parentheses.  The subword of unpaired letters is of the form $i^r (i+1)^s$.  On this subword
\begin{equation*}
	f_i(i^r (i+1)^s) = \begin{cases} i^{r-1} (i+1)^{s+1} & \text{if $r>0$,}\\
	0 & \text{else.}
	\end{cases}
\end{equation*}
All other letters in $w$ remain unchanged; this extends to an operation on semistandard Young tableaux, see for example~\cite{BumpSchilling.2017} 
for more details. There is an edge $b \stackrel{i}{\longrightarrow} b'$  in $B(\lambda)$ whenever $f_i(b)=b'$. See Figure~\ref{figure.B121} for the crystal on 
semistandard Young tableaux of shape $(2,1)$.

There is also the action of the symmetric group on words and tableaux by acting as
\begin{equation}
\label{equation.si}
	s_i(i^r (i+1)^s) = i^s (i+1)^r
\end{equation}
on the subword of unpaired letters and analogously on tableaux.

The \defn{weight} $\mathsf{wt}(b)$ of $b \in \mathsf{SSYT}_n(\lambda)$ is the tuple $(\alpha_1,\alpha_2,\ldots,\alpha_n)$ with $\alpha_j$ 
the number of occurrences of $j$ in $b$. The \defn{character} of the $\mathfrak{sl}_n$-crystal $B(\lambda)$ is the \defn{Schur polynomial}
\[
	 \mathsf{char} B(\lambda) = \sum_{b \in \mathsf{SSYT}_n(\lambda)} x^{\mathsf{wt}(b)} = s_\lambda,
\]
where $x^\alpha = x_1^{\alpha_1}\cdots x_n^{\alpha_n}$.

The \defn{standardization} $\mathsf{std}(b)$ is obtained from $b$ and $\mathsf{wt}(b)$ by replacing the letters 
$i$ in $b$ from left to right by $\alpha_1+\alpha_2+\cdots+\alpha_{i-1}+1,\ldots,\alpha_1+\alpha_2+\cdots+\alpha_i$ for all $1\leqslant i \leqslant n$. 
Then $\mathsf{std}(b)\in \mathsf{SYT}(\lambda)$, the set of standard Young tableaux of shape $\lambda$.

\begin{figure}[t]
\centering
\begin{tabular}{c@{\qquad}c}
 Crystal: & 
 Crystal skeleton:\\[5pt]
\includegraphics{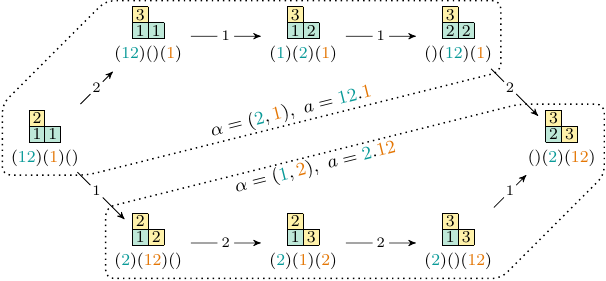}
&
\includegraphics{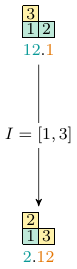}
\end{tabular}
\caption{The $\mathfrak{sl}_3$-crystals $B(2,1) \cong B_{s_1s_2s_1}$ of \S\ref{ss.crystal} and~\S\ref{section.crystal Fw} 
and the corresponding crystal skeletons of~\S\ref{ss.CS} and~\S\ref{ss:CS Fw}.
\label{figure.B121}}
\end{figure}

\subsection{Quasicrystal}
\label{section.quasicrystal}
For a given $\mathfrak{sl}_n$-crystal $B(\lambda)$, the \defn{quasicrystal} $Q_T$ associated to $T \in \mathsf{SYT}(\lambda)$  
is the subgraph of $B(\lambda)$ obtained by restricting to the elements $b\in B(\lambda)$ such that $\mathsf{std}(b) = T$.

\begin{theorem}[\cite{MG.2023}, Theorem 1]
\label{theorem.QC connected}
For $T\in \mathsf{SYT}(\lambda)$, the quasicrystal $Q_T$ is connected in $B(\lambda)$. 
\end{theorem}

For $T \in \mathsf{SYT}(\lambda)$, the letter $i$ is a \defn{descent} if the letter $i+1$ is in a higher row of the tableau 
(in French notation). Denote the descents of $T$ by $d_1<d_2<\cdots<d_k$. The \defn{descent composition} is defined as
$\mathsf{Des}(T) = ( d_1, d_2 - d_1, \ldots, d_k - d_{k-1}, |\lambda|-d_k)$, where $|\lambda|$ is the size of $\lambda$.

In fact~\cite{MG.2023}, two quasicrystals $Q_T$ and $Q_{T'}$ are isomorphic if $\mathsf{Des}(T)=\mathsf{Des}(T')$. Hence quasicrystals
can be indexed by compositions $\alpha$, that is $Q_\alpha := Q_T$ if $\alpha=\mathsf{Des}(T)$.

The edges $f_i$ that change quasicrystal components in $B(\lambda)$ were characterized 
in~\cite{CMRR.2023}.

\begin{proposition}[\cite{CMRR.2023}, Section 2.5.2]
\label{proposition.quasi edges}
Let $b'=f_i(b)$ in $B(\lambda)$ and $b\in Q_T$, $b'\in Q_{T'}$ for $T,T'\in \mathsf{SYT}(\lambda)$.
Then $T\neq T'$ if and only if some $i+1$ is paired with some $i$ in $\mathsf{row}(b)$.
\end{proposition}

The \defn{character} of the quasicrystal $Q_\alpha$ is \defn{Gessel's fundamental quasisymmetric function}~\cite{Gessel.1984}
\[
	\mathsf{char} Q_\alpha = F_\alpha = \sum_{b\in Q_\alpha} x^{\mathsf{wt}(b)}.
\]
Gessel's fundamental quasisymmetric functions $F_\alpha$ form a basis of the vector space $\QSym$ of quasisymmetric functions.  See
for example~\cite{LMvW.2013} for more details.

By construction, the quasicrystals tile the crystal. More precisely, for $n \geqslant |\lambda|$, we have the disjoint union
\[
	B(\lambda) = \bigcup_{T\in \mathsf{SYT}(\lambda)} Q_T.
\]
This is the crystal analogue of~\eqref{equation.s F} as indicated in Figure~\ref{figure.summary}.

\subsection{Crystal skeleton}
\label{ss.CS}
Maas-Gari\'epy~\cite{MG.2023} defined the \defn{crystal skeleton} by contracting the quasicrystal components in the crystal. Since each
quasicrystal component in $B(\lambda)$ contains a unique standard Young tableau $T\in \mathsf{SYT}(\lambda)$, it is natural to index
the vertices of the crystal skeleton $\mathsf{CS}(\lambda)$ by $T \in \mathsf{SYT}(\lambda)$. There is an edge between $T$ and $T'$ in 
$\CS(\lambda)$ if and only if there is an edge between some $b \in Q_T$ and $b' \in Q_{T'}$ in $B(\lambda)$. In~\cite{BCDS.2025}, the 
combinatorics of the crystal skeleton is described, without reference to the original crystal, as follows.

For a permutation $\pi$ and an interval $I$, let $\pi|_I$ be the subword of $\pi$ restricted to the letters in $I$. 

\begin{definition}
\label{definition.Dyck pattern}
Let $T \in \mathsf{SYT}(\lambda)$ with $n=|\lambda|$ and $\pi = \mathsf{row}(T)$.
The interval $I=[i,i+2m] \subseteq [n]$ with $m\geqslant 1$ is a \defn{Dyck pattern interval} for $T $ (resp. $\pi$) if the RSK insertion tableaux of 
$\pi|_I$ and $\pi|_{[i,i+m]}$ have shape $(m+1,m)$ and $(m+1)$, respectively. 
\end{definition}

\begin{example}
The tableau
\[T = \TIKZ[scale=.35]{
\Tableau{{1,3,4},{2,6},{5}}
}\]
has two Dyck pattern intervals: 
$I_1 = [3,5]$ and $I_2 = [2,6]$.
\end{example}

\begin{remark}
\label{remark.outpacing}
Given $T\in \mathsf{SYT}(\lambda)$ and $I=[i,i+2m]$ a Dyck pattern interval of $T$, define the intervals $I^-:=[i,i+m]$ and $I^+:=[i+m+1,i+2m]$, so that 
$I= I^- \cup I^+$. Observe the following:
\begin{enumerate}
\item The condition in Definition~\ref{definition.Dyck pattern} that the RSK insertion tableau of $\pi|_{I^-}$ has shape $(m+1)$ is equivalent to
the statement that the letters in $I^-$ in $T$ form a horizontal strip, meaning there is at most one letter in each column, and the letters in $I^-$ appear in
increasing order left to right. Similarly, the letters in $I^+$ in $T$ also form a horizontal strip and they appear in increasing order left to right.
\item
Furthermore, scanning the columns of $T$ from left to right, there are always weakly more letters from $I^+$ than from $I^-$ due to the fact that
the RSK insertion tableau of $\pi|_I$ has shape $(m+1,m)$. 
\end{enumerate}
\end{remark}

\begin{definition}\label{def:cycle}
With the notation in Definition~\ref{definition.Dyck pattern} and Remark~\ref{remark.outpacing}, we make the following definitions:
\begin{enumerate}
\item 
Associate to every letter in $I^-$ (resp. $I^+$) in $\pi|_I$ a closed (resp. open) bracket.
Since $|I^-|=|I^+|+1$ and by the definition of a Dyck pattern interval, there is precisely one unpaired closed bracket. Let $k$ be the letter in $\pi|_I$ 
associated to the unpaired closed bracket. Note that $k\in I^-$. 
\item
We define $I \cdot T := \cycle(I) \cdot T$, where
\[ 
	\cycle(I)=(m+k,m+k-1,\ldots,k)
\]
is a cycle of length $m+1$ and the action $\cycle(I) \cdot T$ is on the letters in $T$. 
\end{enumerate}
\end{definition}

See~\cite[Section 3.2.2]{BCDS.2025} for further details.

\begin{example}
The interval $I=[6,12]$ is a Dyck pattern interval in the following tableau
\[
	T=\TIKZ[scale=0.35]{
	\filldraw[kstrip!20] (0,2) rectangle ++(1,1);
	\filldraw[kstrip!20] (2,1) rectangle ++(1,1);
	\filldraw[kstrip!20] (3,0) rectangle ++(1,1);
	\filldraw[kstrip!20] (4,0) rectangle ++(1,1);
	\filldraw[kstrip!40] (0,3) rectangle ++(1,1);
	\filldraw[kstrip!40] (3,1) rectangle ++(1,1);
	\filldraw[kstrip!40] (4,1) rectangle ++(1,1);
	\Tableau[\scriptsize]{{1,2,4,8,9},{3,5,7,11,12},{6,13},{10}})}\;.
\]
The letters in $I^-=[6,9]$ shaded in light pink form a horizontal strip as do the letters in $I^+=[10,12]$ shaded in dark pink.
For each column $c$ there are weakly more letters in $I^-$ than in $I^+$ in the columns from $1$ up to $c$. The unpaired letter is 
$k=7$.
\end{example}

\begin{definition}
Let $\lambda \vdash n$ be a partition of $n$. The \defn{crystal skeleton} $\mathsf{CS}(\lambda)$ is the graph with vertices given by
$T \in \mathsf{SYT}(\lambda)$ and an edge $T \stackrel{~I~}{\longrightarrow} T'$  for every Dyck pattern interval $I$ of $T$, where $T'=I \cdot T$.
\end{definition}

\begin{example}
The crystal skeleton $\mathsf{CS}(3,2,1)$ is given in Figure~\ref{figure.CS321}.
\end{example}

\begin{figure}[!h]
\centering
\scalebox{1}{
\includegraphics{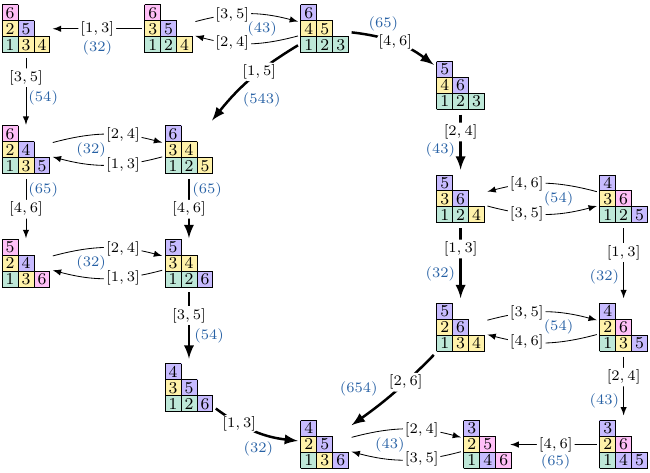}
}
\caption{The crystal skeleton $\mathsf{CS}(3,2,1)$. Edges are labeled by Dyck pattern intervals and cycles; vertices are labeled by standard Young
tableaux in $\SYT(3,2,1)$, and the color indicates the descent composition of each tableau. The $\mathfrak{sl}_3$-subcrystal $B(3,2,1)$ appears with 
bold edges.
\label{figure.CS321}}
\end{figure}

Many intriguing properties of the crystal skeleton were shown in~\cite{BCDS.2025}: 
\begin{enumerate}
\item \defn{Branching property}: Deleting all edges with intervals containing $n$ yields
\begin{equation}
\label{equation.branching}
	\CS(\lambda)_{[1, n-1]} \cong \bigcup_{\mu \in \lambda^-} \mathsf{CS}(\mu),
\end{equation}
where $\lambda^-$ is the set of partitions obtained from $\lambda$ by removing a corner cell.
\item \defn{Subcrystal property}:
The induced subgraph of $\mathsf{CS}(\lambda)$ obtained by including $T\in \mathsf{SYT}(\lambda)$
if the length of $\mathsf{Des}(T)$ is $\ell:=\ell(\lambda)$ is isomorphic to the $\mathfrak{sl}_\ell$-crystal $B(\lambda)$.
\item \defn{Lusztig involution}: The crystal skeleton is invariant under the Lusztig or Sch\"utzenberger involution replacing $T\in \mathsf{SYT}(\lambda)$ by 
its image $\eta(T)$ under the involution and $I=[i,i+2m]$ by $I' = [n+1-i-2m,n+1-i]$.
\item \defn{Promotion}: The graphs obtained by restricting $\mathsf{CS}(\lambda)$ to $[1,n-1]$ and $[2,n]$ are isomorphic related by the promotion
operator.
\item \defn{Dual equivalence subgraph}: The dual equivalence graph is an (undirected) subgraph of $\mathsf{CS}(\lambda)$ by restricting to edges 
labeled by intervals $I$ of length 3.
\end{enumerate}

\section{Quasisymmetric Schur functions}
\label{section.QS}

In this section, we introduce the \defn{Young quasisymmetric Schur functions}, which will help us decompose the crystal skeleton. 
To do so, we first introduce the closely related \defn{quasisymmetric Schur functions}.

The quasisymmetric Schur functions arose from the theory of nonsymmetric Macdonald polynomials, and are closely related to the 
\defn{Young quasisymmetric Schur functions}~\cite{HLMvW.2011,LMvW.2013}. The following material can be found in~\cite{LMvW.2013}. 
Specializing the nonsymmetric Macdonald polynomials to $q=t=0$ produces the \defn{Demazure atoms} $\mathcal{A}_{\gamma}$~\cite{Mas09}.  
The \defn{quasisymmetric Schur functions} can be defined by
\[
	\QS_{\alpha} = \sum_{\gamma} \mathcal{A}_{\gamma},
\]
where the sum is over all weak compositions $\gamma$ such that when the zeros are removed from $\gamma$ the resulting composition is $\alpha$.  
For example, $\QS_{12} = \mathcal{A}_{120}+\mathcal{A}_{102}+ \mathcal{A}_{012}$.  They can also be described combinatorially in terms of tableaux 
(arising from the combinatorial formula for Macdonald polynomials). 

Applying the Lusztig involution to the quasisymmetric Schur functions (with appropriate re-indexing and re-labeling of the variables) produces the 
\defn{Young quasisymmetric Schur functions}, which can also be described combinatorially as fillings of composition diagrams; this is described below. 
The Young quasisymmetric Schur functions are the characters of the quasicrystal skeletons which we will introduce in~\S\ref{section.QCS}.

We now give the combinatorial definition of the  Young quasisymmetric Schur functions.  In the following, if a cell is empty, we think of it as filled with 
the entry infinity, hence any empty cell is considered to be larger than any non-empty cell.

\begin{definition}
\label{def:SSYCT}
A \defn{semistandard Young composition tableau} of shape $\alpha$ is a filling $T$ of the cells of the composition diagram 
(in French notation) of $\alpha$ satisfying the following conditions:
\begin{enumerate}
\item The entries in each row are weakly increasing when read from left to right.
\item The entries in the leftmost column are strictly increasing when read from bottom to top.
\item If $r'>r$ with $(r,c+1)$ in the diagram of $\alpha$ and $T(r,c+1) \geqslant T(r',c)$, then $T(r,c+1) >T(r',c+1)$.
\end{enumerate}
The set of semistandard Young composition tableaux of shape $\alpha$ is denoted by $\mathsf{SSYCT}(\alpha)$.
\end{definition}

A \defn{standard Young composition tableau} (SYCT) of shape $\alpha$ with $|\alpha|=n$ is a semistandard Young composition tableau in which 
each entry from $[n]$ appears exactly once. We denote the set of standard Young composition tableaux of shape $\alpha$ by
$\mathsf{SYCT}(\alpha)$. 

\begin{example}
Examples of standard Young composition tableaux are the vertices in the graph in Figure~\ref{figure.QCS321}.
\end{example}

The third condition in Definition~\ref{def:SSYCT} can be interpreted as a \defn{triple condition}, depicted as follows, where $x=T(r',c)$, $y=T(r',c+1)$,
and $z=T(r,c+1)$:
\begin{equation}
\label{equation.triple}
\begin{matrix}
\text{Semistandard tableaux:} &\qquad & \text{Standard tableaux:}\\[3pt]
 \TIKZ[scale=.5]{
	\draw (0,2.5) rectangle +(2,1);
		\draw (1,2.5) to +(0,1);
		\node at (0.5, 3) {$x$\strut};
		\node at (1.5, 3) {$y$\strut}; 
	\draw (1,0) rectangle +(1,1);
    		\node at (1.5, .5) {$z$\strut};  
	\draw (1.5, 1) to 
		node[fill=white, sloped, inner sep = .5pt]{\footnotesize$\leqslant$} 
		+(-1, 1.5);
	\draw[densely dashed] (1.5, 1) to
    +(0, 1.5);
}&&
 \TIKZ[scale=.5]{
	\draw (0,2.5) rectangle +(2,1);
		\draw (1,2.5) to +(0,1);
		\node at (0.5, 3) {$x$\strut};
		\node at (1.5, 3) {$y$\strut}; 
	\draw (1,0) rectangle +(1,1);
    		\node at (1.5, .5) {$z$\strut};  
	\draw (1.5, 1) to 
		 node[fill=white, sloped, inner sep = .5pt]{\footnotesize$<$} 
		+(-1, 1.5);
	\draw[densely dashed] (1.5, 1) to
        +(0, 1.5);
}\\
\text{If $z \geqslant x$, then $z>y$.} && \text{If $z>x$, then $z>y$.}
\end{matrix}
\end{equation}
In particular, the condition ensures that as long as the cells containing $z$ and $x$ exist, the cell containing $y$ must also exist. 
  
Standard Young composition tableaux are in bijection with saturated chains in the 
\defn{Young composition poset}~\cite{LMvW.2013} whose elements are compositions such that $\alpha=(\alpha_1,\alpha_2, \hdots , \alpha_{\ell})$ is covered by
\begin{enumerate}
\item $(\alpha_1,\alpha_2, \hdots , \alpha_{\ell},1)$
\item $(\alpha_1, \alpha_2, \hdots , \alpha_i+1, \hdots , \alpha_{\ell})$ provided $\alpha_j \not= \alpha_i$ for all $j>i$.
\end{enumerate}

In light of the covering relations described above, we now describe which cells in a Young composition tableau of shape $\alpha$ (with $|\alpha|=n$) 
can contain the entry $n$.

\begin{definition}
Let $\alpha$ be a composition of $n$.  We call the cell $c$ in the diagram of $\alpha$ a \defn{removable cell} in the diagram of $\alpha$ if there exists a 
Young composition tableau with entry $n$ in cell $c$.
\end{definition}

The covering relations in the Young composition poset imply that a cell $(r,c)$ is a removable cell of $\alpha=(\alpha_1, \alpha_2, \hdots , \alpha_{\ell})$ 
if and only if either ($c=1$ and $r=\ell$) or ($c=\alpha_r \not=1$ and $\alpha_i \not= \alpha_r-1$ for all $i > r$).

\begin{definition}
The \defn{superstandard filling} of a composition shape $\alpha$ (where $| \alpha| = n$) is the filling obtained by placing the entries from $[n]$ 
consecutively into the cells of $\alpha$ starting from the bottom and moving left to right along rows, going from bottom row to top row.
\end{definition}

\begin{lemma}
\label{lem:superstandard}
The superstandard filling of a composition $\alpha$ is always a Young composition tableau of shape $\alpha$.
\end{lemma}

\begin{proof}
Conditions 1 and 2 in Definition~\ref{def:SSYCT} are satisfied by construction.  Condition~3 follows from the fact that if $r'>r$ then $T(r,c+1) < T(r',c)$. 
Therefore the superstandard filling of a composition shape is indeed a Young composition tableau.
\end{proof}

\begin{definition}
The \defn{Young quasisymmetric Schur function} $\YQS_{\alpha}$ is the polynomial generated by the set of semistandard Young composition tableaux
$\mathsf{SSYCT}(\alpha)$ of shape $\alpha$
 \[
 	\YQS_{\alpha} = \sum_{T \in \mathsf{SSYCT}(\alpha)} x^{\wt(T)},
\]
where $\wt(T)$ is the weight of $T$.
\end{definition}

The decomposition of the Young quasisymmetric Schur functions into fundamental quasisymmetric functions can be described in terms of descent sets
on standard Young composition tableaux. If $T\in \mathsf{SYCT}(\alpha)$ with $\alpha \models n$, we define the \defn{descent set} as
\[
	\mathsf{des}(T) = \{ i \mid i+1 \textrm{ appears weakly left of } i \} \subseteq [n-1].
\]
From this, we define the \defn{descent composition} $\mathsf{Des}(T)$ as in~\S\ref{section.quasicrystal}. Then
\begin{equation}
\label{equation.YQS F}
	\YQS_{\alpha} = \sum_{\beta} d_{\alpha, \beta} F_{\beta},
\end{equation}
where $d_{\alpha,\beta}$ is the number of $T \in \mathsf{SYCT}(\alpha)$ such that $\mathsf{Des}(T)=\beta$.

\begin{example}
Computing the descent compositions of the standard Young composition tableaux of shape $\alpha=(2,3,1)$ in Figure~\ref{figure.QCS321},
we find
\[
	\YQS_{(2,3,1)} = F_{(1,3,1,1)} + F_{(1,2,2,1)} + F_{(2,3,1)} + F_{(1,2,1,2)} + F_{(2,2,2)}.
\]
\end{example}

We now describe a bijection between standard Young tableaux and Young composition tableaux; this will be essential to our study of the crystal skeleton, 
whose vertices are indexed by standard Young tableaux of a fixed shape $\lambda$.

\begin{definition}\cite[Section 4.3]{LMvW.2013}
\label{def.phi}
Let $\lambda \vdash n$ be a partition of $n$. Define the following bijection
\[
\begin{split}
        \varphi \colon \bigcup_\alpha \mathsf{SYCT}(\alpha) &\to \mathsf{SYT}(\lambda),\\
	\yc{T} &\mapsto T,
\end{split}
\]
where the union is over all compositions $\alpha$ that rearrange $\lambda$. 
\begin{enumerate}
\item
Given $\yc{T} \in \mathsf{SYCT}(\alpha)$, the standard tableau $T=\varphi(\yc{T})$ is obtained by simply arranging all the columns in increasing 
order from bottom to top. 
\item
The inverse map $\varphi^{-1}(T)$ is defined as follows. Start by writing the leftmost column of $T$ in increasing order from bottom to top.  
Then place the entries in the second column of $T$, starting with the smallest and moving to the largest, so that each entry is placed into the 
highest unoccupied cell whose neighbor immediately to the left is smaller. Repeat with all subsequent columns, left to right.
\end{enumerate}
Henceforth, we will set $\yc{T}:= \varphi^{-1}(T)$.
\end{definition}

\begin{example}
The map $\varphi$ is given by mapping the Young composition tableaux in Figure~\ref{figure.QCS321} to the corresponding standard Young tableau 
in Figure~\ref{figure.CS321}.
\end{example}

\begin{remark}
\label{remark.consecutive order}
Since an entry can never be placed to the right of a larger entry, the entries can be placed into the diagram in consecutive order rather 
than column by column.
\end{remark}

The bijection $\varphi$ proves~\eqref{equation.s YQS}. Note that this decomposition and the upper triangularity property of these as bases for $\QSym$ 
together imply that any symmetric function that can be written as a positive sum of Young quasisymmetric Schur functions will automatically be Schur positive.

\section{Quasicrystal skeletons}
\label{section.QCS}

This section introduces the main new object of this paper, the quasicrystal skeleton.
We begin in \S\ref{section.definition QS} with the definition of the quasicrystal skeleton. 
In \S\ref{section.characterization edges}, specifically Theorem~\ref{theorem.component change}, 
we characterize the edges which change components in the quasicrystal skeleton. 
The proof of Theorem~\ref{theorem.component change} is given in \S\ref{section.proof}.
In \S\ref{section.connectedness} we address connectivity of
quasicrystal skeletons. In \S\ref{section.Bruhat}, we conclude by analyzing the contraction of
quasicrystal skeletons to obtain quotients of Bruhat order.

\subsection{Definition of quasicrystal skeleton}
\label{section.definition QS}

\begin{definition}
Let $\lambda \vdash n$ be a partition and let $\alpha$ be a composition which is a rearrangement of the parts of $\lambda$.
We define the \defn{quasicrystal skeleton} $\mathsf{QCS}_\alpha$ as the induced subgraph of $\mathsf{CS}(\lambda)$ consisting of all vertices
$T \in \mathsf{SYT}(\lambda)$ such that $\varphi^{-1}(T)$ has shape $\alpha$.
\end{definition}

\begin{figure}[t]
\centering
\includegraphics{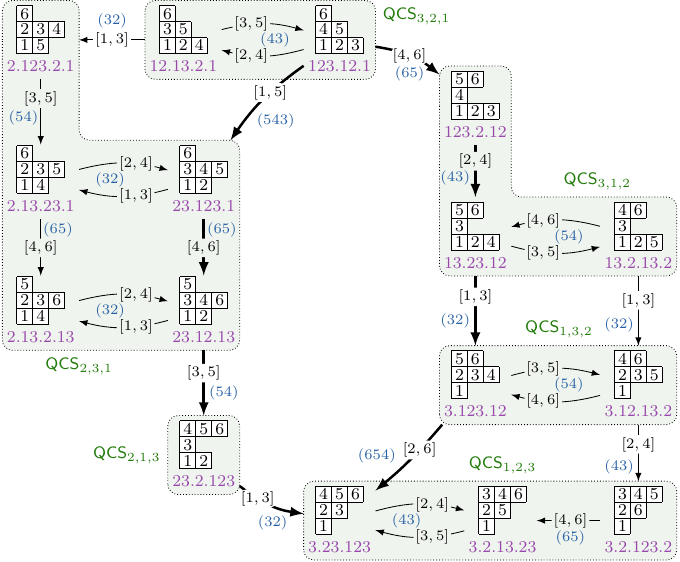}
\caption{The crystal skeleton $\mathsf{CS}(3,2,1)$ partitioned into quasicrystal skeletons, with vertices labeled by two combinatorial models: 
Young composition tableaux (see \S\ref{section.QS}) and reduced words for the permutation $w^{-1} =  s_1s_2s_3s_1s_2s_1$ (see \S\ref{ss:CS Fw}). 
Compare to Figure~\ref{figure.CS321}.
\label{figure.QCS321}}
\end{figure}

\begin{remark}
Since $\varphi$ of Definition~\ref{def.phi} is a bijection, we can either label the vertices of $\mathsf{QCS}_\alpha$ by standard Young tableaux
of shape $\lambda$ or alternatively by standard Young composition tableaux of shape $\alpha$.
\end{remark}

In Figure~\ref{figure.QCS321}, the crystal skeleton $\mathsf{CS}(3,2,1)$ is given with standard Young composition tableaux as vertices. The
quasicrystal skeletons within the crystal skeleton are shaded.

Note that by~\eqref{equation.YQS F} the character of $\mathsf{QCS}_\alpha$ is the Young quasisymmetric Schur function
\[
	\mathsf{char} \mathsf{QCS}_\alpha = \YQS_\alpha.
\]

\subsection{Characterization of quasicrystal skeleton edges}
\label{section.characterization edges}
Our main result characterizes the edges in the crystal skeleton $\CS(\lambda)$ that change quasicrystal skeleton components.

Recall that in the crystal skeleton, edges are labeled by Dyck pattern intervals (see Definition~\ref{definition.Dyck pattern}), and that we use 
$k$ to denote the bracketed letter in the word $\pi|_{I}$ (see Definition~\ref{def:cycle}).

\begin{theorem}
\label{theorem.component change}
Let $T \stackrel{~I~}{\longrightarrow} T'$ where $I=[i,i+2m]$ is an edge in $\mathsf{CS}(\lambda)$ and let 
$\cycle(I) = (m+k,m+k-1,\ldots,k)$ be the cycle such that $T' = \cycle(I) \cdot T$.
Then $T \in \mathsf{QCS}_\alpha$ and $T' \in \mathsf{QCS}_\beta$ with $\alpha \neq \beta$ if and only if
\begin{enumerate}[label=(\Roman*), ref={Condition \Roman*}]
\item \label{condition1} The letters in $A:=[i,k]$ occur in the first  $|A|=k-i+1$ columns of $T$ and $k>i$; and
\item \label{condition2} If $i$ occurs in row $r_0$ of $T$, then $\alpha_{r_0+1}<\alpha_{r_0}$.
\end{enumerate}
\end{theorem}

Theorem~\ref{theorem.component change} is the analogue of Proposition~\ref{proposition.quasi edges} for quasicrystals and is proved in
\S\ref{section.proof}. We refer to Theorem~\ref{theorem.component change}~(I) as \ref{condition1} and
Theorem~\ref{theorem.component change}~(II) as \ref{condition2}. 

\begin{remark}
\label{YCTrows}
\ref{condition1} implies that the first $|A| = k-i+1$ columns in rows $r_0$ and $r_0+1$ in $\varphi^{-1}(T)$ begin with
\[
\TIKZ[xscale=1.25, scale=.7, font=\footnotesize]{
\filldraw[IminColor!10] (0,0) rectangle (5,1);
\filldraw[IpluColor!10] (0,1) rectangle (4,2);
\begin{scope}[thick]
\draw
	(0,0) to (5,0)
	(0,1) to (5,1)
	(0,2) to (4,2);
\foreach \h [count=\x from 0] in {2, 2, 2, 2, 2, 1}{\draw (\x, 0) to (\x,\h);}
\end{scope}
\foreach \c [count=\x from 1] in {i, i\! +\! 1, \cdots, k \sm 1, k}
	{\node[IminColor!70!black] at (\x-.5, .5) {\strut$\c$};}
\node[Bv, IpluColor!70!black] (i+m+1) at (1-.5, 1.5) {};
	\draw[<-,  IpluColor!70!black, bend right=20] (i+m+1) to +(-.3,1) node[left, inner sep=2pt]{$i+m+1$};
\node[Bv, IpluColor!70!black] (i+m+2) at (2-.5, 1.5) {};
	\draw[<-,  IpluColor!70!black] (i+m+2) to +(0,.8) node[above, inner sep=2pt]{$i+m+2$};
\node[IpluColor!70!black] at (3-.5, 1.5) {\strut$\cdots$};
\node[IpluColor!70!black] at (4-.5, 1.5) {\strut$k\! +\! m$};
\node[right] at (4, 1.5){$\cdots$};
\node[right] at (5, .5){$\cdots$};
\node[left] at (0, 1.5){row $r_0+1$:};
\node[left] at (0, .5){row $r_0$:};
}\]
\end{remark}

One can use Theorem~\ref{theorem.component change} to further describe the edges between different quasicrystal skeleton components.

\begin{corollary}
\label{corollary.component change}
Let $T \stackrel{~I~}{\longrightarrow} T'$ with $I=[i,i+2m]$ an edge in $\mathsf{CS}(\lambda)$.
Suppose $T \in \mathsf{QCS}_\alpha$ and $T' \in \mathsf{QCS}_\beta$ with $\alpha \neq \beta$. Then:
\begin{enumerate}
\item $i = 1$ or $i-1$ is a descent in $T$; and
\item $I$ is length-preserving: the lengths of $\mathsf{Des}(T)$ and $\mathsf{Des}(T')$ are equal.
\end{enumerate}
\end{corollary}

\begin{proof}
The first point follows immediately from $i$ being in the first column of $T$ by Theorem~\ref{theorem.component change}.

For the second point, we use \cite[Lem.\ 4.25, Thm.\ 4.26]{BCDS.2025}: $I$ is length-preserving if and only if $T|_{[i-1, i+2m]}$ and $T|_{[i, i+2m + 1]}$ 
do not jeu de taquin to rectangles. The fact that $T|_{[i-1,i+2m]}$ does not jeu de taquin to a rectangle follows from the fact that $k>i$ 
by Theorem~\ref{theorem.component change} and that there is a descent at $i-1$.  
For the latter, let $\pi = \row(T)$. Then~\cite[Lem.\ 4.25]{BCDS.2025} states that $T|_{[i, i+2m + 1]}$ is a rectangle if and only if 
\[\pi|_{\{i+m, i+2m, i+2m+1\}} = i+2m\ \hspace*{.2in} i+2m+1\ \hspace*{.2in} i+m\] and $\cycle(I) = (i+2m, i+2m-1, \dots, i+m)$. 
This means that under the assumption that $\alpha\neq \beta$, $I$ is length-increasing if and only if $k = i+m$ and $i + 2m + 1$ appears in the 
same column as $k$ in $T$ since by Theorem~\ref{theorem.component change} the letters in $A=[i,i+m]$ appear in the first $|A|$ columns. 
But in this case, the row containing $i+m+1$ in $\varphi^{-1}(T)$ is at least as long as the row containing $i$, so that $I$ would not be shape-changing.
\end{proof}

\subsection{Proof of Theorem~\ref{theorem.component change}}
\label{section.proof}

The structure of the proof of  Theorem~\ref{theorem.component change} is as follows. Henceforth, we use the notation $\yc{T}:=\varphi^{-1}(T)$ and
$I \cdot \yc{T}:=\varphi^{-1}(I \cdot T)$ for $T \in \SYT(\lambda)$ and $I$ a Dyck pattern interval in $T$.
\begin{enumerate}
    \item We show in \S\ref{ss:whenImovesk} that \ref{condition1} is equivalent to a local condition on the locations of the box containing $k$ in 
    $\yc{T}$ and the box containing $k+m$ in $I \cdot \yc{T}$. We refer to this condition as $I$ \emph{moving the cell of} $k$; see Definition~\ref{def:Imovesk}.
    \item Let $\yc{T}|_{J}$ be the restriction of $\yc{T}$ to the interval $J = [1,i+2m]$. We next show in \S\ref{ss:smallerintervalshapechange} 
    that $I$ moves the cell of $k$ if and only if the shape of $\yc{T}|_J$ and $I \cdot \yc{T}|_J$ differ by a specific row swap.
    \item Finally, we prove Theorem~\ref{theorem.component change} in \S\ref{ss:proofofcomponentchange} by showing that if \ref{condition1} is 
    met, then \ref{condition2} is met if and only if the shapes of $\yc{T}$ and $I \cdot \yc{T}$ differ.
\end{enumerate}

Before we embark on proving the above, we begin with some preliminary lemmas that study how the Dyck pattern interval $I$ interacts with the 
Young composition tableaux $\yc{T}$ and $I \cdot \yc{T}$. These concepts will be broadly applicable throughout our arguments. 

\subsubsection{Preliminary lemmas}
\label{ss:preliminarylemmas}

Our first result describes where elements in a Dyck pattern interval are placed in a Young composition tableau. 
Recall from Remark~\ref{remark.outpacing} that in a standard tableau $T\in \mathsf{CS}(\lambda)$ with Dyck pattern interval $I=[i,i+2m]$ 
 there are always weakly more letters in $I^+=[i+m+1,i+2m]$ than in $I^-=[i,i+m]$ (scanning from left to right). Since this is a statement about the columns
of $T$ and the set of letters in columns does not change under $\varphi^{-1}$, the statement remains true for the corresponding 
Young composition tableau $\yc{T}=\varphi^{-1}(T)$.

Throughout this section, we will fix the following notation: let $T \stackrel{~I~}{\longrightarrow} T'$ with $I=[i,i+2m]$ be an edge in $\mathsf{CS}(\lambda)$ 
and let $\cycle(I) = (m+k,m+k-1,\ldots,k)$ be the cycle such that $T' = \cycle(I) \cdot T$. For every $x \in [n]$, the action of $\cycle(I)$ on $x$ will be 
denoted $x':= \cycle(I) \cdot x$.

\begin{lemma}
\label{lemma.column separation}
Let $T \stackrel{~I~}{\longrightarrow} T'$ with $I=[i,i+2m]$ an edge in $\mathsf{CS}(\lambda)$ and let $\cycle(I) = (m+k,m+k-1,\ldots,k)$ be the cycle 
such that $T' = \cycle(I) \cdot T$. Then each entry of $\cycle(I)$ appears in a different column of $T$.
\end{lemma}

\begin{proof}
Since $I$ is a Dyck pattern interval, by Remark~\ref{remark.outpacing} the letters in the intervals $I^-=[i,i+m]$ and $I^+=[i+m+1,i+2m]$ each appear in 
increasing order in $\pi|_I$,  where $\pi=\row(T)$ by Definition~\ref{definition.Dyck pattern}.
Furthermore, by Remark~\ref{remark.outpacing}, the entries in $I^-$ (resp. $I^+$) must appear in distinct columns of $T$.
Since $k$ is unpaired in the bracketing of the letters in $I^-$ and $I^+$ in $\pi|_I$ but all other entries are bracketed, the 
entries in $[i,\hdots, k-1]$ must be bracketed with the first $k-i$ entries from $I^+$, which must therefore appear in a higher row of $T$.  
These are precisely the entries in the interval $[i+m+1,m+k]$, which therefore appear in distinct columns to the left of the column containing $k$.  

Similarly, the entries in the interval $[k+1,i+m]$ are all paired with entries from $[i+m+1,i+2m]$ occurring to the right of the column containing $k$.  
We already know that the entries in $[k+1,i+m]$ must appear in distinct columns, and since they appear to the right of the column containing $k$, 
none of these can appear in the same column as any entry from $[i+m+1,m+k]$.  Therefore the entries in the interval 
$\{k\} \cup [k+1,i+m] \cup [i+m+1,m+k]=[k,m+k]$ must all appear in distinct columns.
\end{proof}

Next we introduce some useful notation. Let $\shape(\yc{T})$ be the (composition) shape of $\yc{T}=\varphi^{-1}(T)$ for $T\in \mathsf{SYT}(\lambda)$. 
For any interval $J \subseteq [1,n]$, let $\shape(\yc{T}|_J)$ be the arrangement of boxes containing elements of $J$ in $\yc{T}$.
For $x \in [1,n]$, if $x$ appears in row $r$ and column $c$ of $\yc{T}$, we write
\[
	\cell_{\yc{T}}(x) = (r,c), \quad \row_{\yc{T}}(x) = r, \quad  \col_{\yc{T}}(x)=c, 
	\quad \text{and} \quad 
	\yc{T}(r,c) = x.
\]
For a fixed column $c$, we write
\[ 
	\Col_{\yc{T}}(c):= \{ x \in [n] \mid x \textrm{ is in column } c \textrm{ in } \yc{T} \}. 
\]

\begin{definition}\label{definition.columnstable}
    A column $c$ of $\yc{T}$ is \defn{stable under $I$} if 
\[
	\cell_{\yc{T}}(x) = \cell_{I \cdot \yc{T}}(\cycle(I) \cdot x) 
\]
for all $x \in \Col_{\yc{T}}(c)$.
\end{definition}

For example, in Figure~\ref{figure.QCS321}, we have the edge
\[\TIKZ[font=\footnotesize]{
\node (T-hat) at (0,0) {
	$\yc{T} =$
	\TIKZ[scale=0.35]{
	\filldraw[kstrip!20] (2,0) rectangle ++(1,1);
	\Tableau[\scriptsize]{{1,2,3},{4,5},{6}})
	}};
\node (15-T-hat) at (3,0) {
	\TIKZ[scale=0.35]{
	\filldraw[kstrip!20] (2,1) rectangle ++(1,1);
	\Tableau[\scriptsize]{{1,2},{3,4,5},{6}})
	}};
\draw[->] 
	(T-hat) to node[above, sloped] {$[1,5]$} (15-T-hat);
}\]
with $I=[1,5]$ and $\cycle(I)=(543)$.
Here, columns 1 and 2 are stable, but column 3 is not.

For each of the following lemmas, 
we define $T$, the interval $I = [i, i+2m]$, and $k$ as in Theorem~\ref{theorem.component change}.

The next lemma, which we refer to as the \emph{Column-Stability Lemma} (Lemma~\ref{lem_dec_by_one}), is useful for induction on columns that 
are largely unaffected by the application of $I$. 

\begin{lemma}[Column-Stability Lemma]
\label{lem_dec_by_one}
Fix an interval $J=[1,j]\subseteq [1,n]$. Suppose that $\Col_{\yc{T}|_J}(c-1)$ is stable under $I$ and $k \not \in \{ \Col_{\yc{T}|_J}(c-1), \Col_{\yc{T}|_J}(c) \}$. 
Then $\Col_{\yc{T}|_J}(c)$ is also stable under $I$.
\end{lemma}

\begin{proof}
Let $x' = \cycle(I) \cdot x$ for every entry $x$ in $T$. By Lemma~\ref{lemma.column separation}, at most one entry is changed by $I$ in each of 
$\Col_{\yc{T}|_J}(c)$ and $\Col_{\yc{T}|_J}(c-1)$; and since neither contains $k$, each entry $y\in \Col_{\yc{T}|_J}(c-1) \cup \Col_{\yc{T}|_J}(c)$ has 
$y' = y$ or $y' = y-1$. Hence, for any two $y, z \in \Col_{\yc{T}|_J}(c-1) \cup \Col_{\yc{T}|_J}(c)$, we have  $z<y$ if and only if $z'<y'$. 

Now, let $y$ be the smallest entry in column $\Col_{\yc{T}|_J}(c)$. Then $y'$ is also the smallest entry in $\Col_{I \cdot \yc{T}|_J}(c)$, so that $y'$ lies in 
the highest row of $I \cdot \yc{T}$ that is immediately to the right of an entry smaller than $y'$. Therefore $\row_{I \cdot \yc{T}}(y')=\row_{\yc{T}}(y)$ since 
by assumption column $\Col_{\yc{T}|_J}(c-1)$ is stable under $I$. 

The same argument carries through for the remainder of the entries in column $\Col_{\yc{T}|_J}(c)$, ignoring the rows of $I \cdot \yc{T}$ 
containing a smaller entry in column $\Col_{\yc{T}|_J}(c)$. Therefore for each $x$ in column $\Col_{\yc{T}|_J}(c)$, we have  
$\row_{I \cdot \yc{T}}(x')=\row_{\yc{T}}(x)$ and thus column $\Col_{\yc{T}|_J}(c)$ is stable under $I$. 
\end{proof}

\begin{lemma}\label{lemmaI:skew shape stability}
    Let $T,S \in \SYT(\lambda)$ and  $J = [1,j] \subseteq [1,n]$ be an interval. If $T$ and $S$ are identical outside of $J$ and 
    $\shape(\yc{T}|_J) = \shape(\yc{S}|_J)$, then $\shape(\yc{T}) = \shape(\yc{S})$.
\end{lemma}

\begin{proof}
By assumption
\[
	\alpha := \shape(\yc{T}|_{J}) = \shape(\yc{S}|_{J}). 
\]
It remains to place the boxes in $[j+1,n]$ in both $\yc{T}$ and $\yc{S}$. By assumption $T|_{[j+1,n]} = S|_{[j+1,n]}$ (as skew tableaux), 
so $\yc{T}$ and $\yc{S}$ have the same column sets on $[j+1,n]$. We place box $j+1$ in both $\yc{T}|_{[1,j]}$ and $\yc{S}|_{[1,j]}$ under $\varphi^{-1}$,
and since $j+1$ is larger than any element in $[1,j]$, it is placed in the same position in both composition tableaux since they both have shape $\alpha$. 
We continue in this method for the remaining boxes $j+2, \ldots, n$, where each box is placed in the same position in both composition tableaux. This 
yields the claim.
\end{proof}

\subsubsection{When $I$ moves the cell of $k$}
\label{ss:whenImovesk}

We now wish to explicitly describe how the action of $I$ moves $\cell_{\yc{T}}(k)$ in $\yc{T}$. To do so, we introduce some language to describe 
the relationship between intervals in Young composition tableaux. 

\begin{definition}\label{def:Imovesk}
    A Dyck pattern interval $I$ is said to \defn{move the cell of $k$} if $\cell_{\yc{T}}(k) \neq \cell_{I \cdot \yc{T}}(\cycle(I)\cdot k)$. 
    Otherwise, $I$ does not move the cell of $k$.
\end{definition}

\begin{example}
Consider the following subgraph of Figure~\ref{figure.QCS321}.
\[\TIKZ[font=\footnotesize]{
\node (T-hat) at (0,0) {
	$\yc{T} =$
	\TIKZ[scale=0.35]{
	\filldraw[kstrip!20] (2,0) rectangle ++(1,1);
	\Tableau[\scriptsize]{{1,2,3},{4,5},{6}})
	}};
\node (15-T-hat) at (3,.75) {
	\TIKZ[scale=0.35]{
	\filldraw[kstrip!20] (2,1) rectangle ++(1,1);
	\Tableau[\scriptsize]{{1,2},{3,4,5},{6}})
	}};
\node (24-T-hat) at (3,-.75) {
	\TIKZ[scale=0.35]{
	\filldraw[kstrip!20] (2,0) rectangle ++(1,1);
	\Tableau[\scriptsize]{{1,2,4},{3,5},{6}})
	}};
\draw[->, shorten <= -3pt, bend left=10] 
	(T-hat) to node[above, sloped] {$[1,5]$} (15-T-hat);
\draw[->, bend right=10] 
	(T-hat) to node[below, sloped] {$[2,4]$} (24-T-hat);
}
\]
Both edges out of $\yc{T}$ have $k=3$, one that moves the cell of $k$ and one that does not. Specifically, the interval 
$I=[1,5]$ moves the cell of $3$ since $\cell_{\yc{T}}(3) =(1,3) \neq (2,3) = \cell_{I \cdot \yc{T}}(5)$ and $\cycle([1,5]) = (543)$. 
The Dyck pattern interval $I=[2,4]$ does not move the cell of $3$ since $\cell_{\yc{T}}(3) =(1,3) = \cell_{I \cdot \yc{T}}(4)$ and $\cycle([2,4]) = (43)$.
\end{example}

Our goal is to prove the following.

\begin{proposition}
\label{prop:kmovesiffcondition1}    
Let $T\in \SYT(\lambda)$ with Dyck pattern interval $I$ and $\cycle(I) = (m+k, m+k-1, \ldots, k)$. Then $I$ moves the cell of $k$ 
if and only if \ref{condition1} in Theorem \ref{theorem.component change} is satisfied. 
\end{proposition}

To prove Proposition~\ref{prop:kmovesiffcondition1}, we first establish a technical lemma.

\begin{lemma}\label{lemma:kmoves1}
Let $\cell_{\yc{T}}(k) = (r_0,c_0)$. 
The Dyck pattern interval $I$ moves the cell of $k$ if and only if 
    \[\col_{\yc{T}}(k-1) = \col_{\yc{T}}(k+m) = c_0-1
    \quad\text{and}\quad
    \row_{\yc{T}} (k+m) > r_0.
    \]
Note that the second condition is never satisfied if $c_0 = 1$. 
\end{lemma}

\begin{proof}
Since the leftmost column is always stable under $I$ and $\col_{\yc{T}}(k) \notin [1,c_0-1]$, the Column-Stability Lemma 
(Lemma~\ref{lem_dec_by_one}) implies that columns $1, \ldots , c_0-1$ are stable under $I$.  

\smallskip

\noindent $\Leftarrow$: 
Assume first that $\col_{\yc{T}}(k-1) = \col_{\yc{T}}(k+m) = c_0-1$ and $\row_{\yc{T}} (k+m) > r_0$.  Let $r = \row_{\yc{T}} (k+m)$.
\[{\def\Height{3}
\TIKZ[scale=.85]{
\node[left] at (-.5,.5*\Height) {$\yc{T}: $};
\filldraw[black!20] (0,0) rectangle ++(-.5,\Height);
\filldraw[pumpkin!20] (0,0) rectangle ++(1,\Height);
\filldraw[rose!20] (1,0) rectangle ++(1,\Height);
\foreach \x in {0, 1, 2}{\draw[densely dashed, black!60] (\x,0) to ++(0,\Height);}
\draw[fill=white] (1.0,.5)  rectangle node {$k$} ++(1,1);
\node[minimum size=.75cm, inner sep=0pt, draw, fill=white] (k+m) at (.5,2.25) {};
	\node at (k+m) {\footnotesize $k \! + \! m$};
	\draw[->] (k+m.north) to +(0, .25);
	\draw[->] (k+m.south) to +(0, -.25);
\node[minimum size=.75cm, inner sep=0pt, draw, fill=white] (k-1) at (.5,.75) {};
	\node at (k-1) {\footnotesize $k\!-\!1$};
	\draw[->] (k-1.north) to +(0, .25);
	\draw[->] (k-1.south) to +(0, -.25);
\draw[thick] (0, 1.5) to ++(2.6,0); 
\node[above, inner sep=1pt, pumpkin!80!black] at (.5, \Height) {\begin{tabular}{c}col\\[-3pt]$c_0 \sm 1$\end{tabular}};
\node[above, inner sep=1pt, rose!80!black] at (1.5, \Height) {\begin{tabular}{c}col\\[-3pt]$c_0$\end{tabular}};
\draw[thin, <-, black!80] (2,1) to ++(.3,0) node[right]{row $r_0$\strut};
\draw[thin, <-, black!80] (k+m.east) to +(1.35, 0) node[right]{row $r$\strut};
}}\]
Then, by column stability, $\cell_{I \cdot \yc{T}}(k+m-1) = \cell_{\yc{T}}(k+m) = (r, c_0-1)$. We claim that $k+m$ lands in row $r$ in $I \cdot \yc{T}$
when applying $\varphi^{-1}$ to $I \cdot T$: first, $k+m$ cannot be higher in $I \cdot \yc{T}$ since (1) $k$ is not higher in $\yc{T}$ and 
(2) all cells in column $c_0-1$ of $I\cdot \yc{T}$ above row $r$ are greater than the entries in the interval $I$. Conversely,  if $k+m$ appears lower than row $r$ 
in $I \cdot \yc{T}$, then the cells containing $k+m-1$ and $k+m$, together with the cell immediately to the right of $k+m-1$, 
would violate the triple condition~\eqref{equation.triple}. 
\[ \TIKZ[scale=.75]{
	\draw (0,2) rectangle +(2,1);
		\draw (1,2) to +(0,1);
		\node[Bv] (k+m-1) at (0.5, 2.5) {};
			\draw[rounded corners, ->] (.5, 3.25) node[above, inner sep = 1pt] {\footnotesize $k \! + \! m \! - \! 1$} to (k+m-1);
		\node at (1.5, 2.5) {?}; 
	\draw (1,0) rectangle +(1,1);
    		\node[] (k+m) at (1.5, .5) {\scriptsize $k \! + \! m$};  
	\draw (1.5, 1) to 
		node[fill=white, sloped, inner sep = .5pt]{\footnotesize$\leqslant$} 
		+(-1, 1);
	\draw[densely dashed, pumpkin, very thick] (1.5, 1) to +(0, 1);
	\node[right] at (2,2.5) {(row $r$)};
	\node[left, xshift=-5pt] at (0,1.5) {in $I \cdot \yc{T}$:};
}\]
Therefore $\row_{I \cdot \yc{T}}(k+m) = r>r_0$ so that 
$\cell_{I \cdot \yc{T}}(k+m) \not= (r_0,c_0)$.

\medskip

\noindent $\Rightarrow$:
For the opposite direction, assume $\cell_{I \cdot \yc{T}}(k+m) \not= (r_0,c_0)$, and recall that columns $1, \dots, c_0 - 1$ are stable under $I$. Further 
recall that values in column $c_0$ of $\yc{T}$ and $I\cdot \yc{T}$ outside of $I$ are identical, and the only entry from $I$ in column $c_0$ of
$I\cdot \yc{T}$ is $k+m$. For entries $x< \min(I)$ in column $c_0$, stability of column $c_0-1$ implies $\cell_{I \cdot \yc{T}}(x)=\cell_{\yc{T}}(x)$. 
Therefore, $(I \cdot \yc{T})(r, c_0) \geqslant \min(I)$. By $k$'s placement in $\yc{T}$, together with stability, we also know $I\cdot \yc{T}(r,c_0-1)<k$. 
Hence, $k+m$ has $(r_0,c_0)$ available in $I \cdot \yc{T}$, and our assumption that $k+m$ \emph{is not} placed here implies that it is placed strictly 
higher (i.e.\ $\row_{I \cdot \yc{T}}(k+m) > r_0$). This means there exists an entry $a$ in column $c_0-1$ of $\yc{T}$ in a row above row $r_0$ such that 
$k <a<k+m$. By properties of Dyck pattern intervals, (1) the only possible value to meet these conditions is $k+m$, and (2) $k+m$ occurring in 
column $c_0-1$ implies $k-1$ is also in column $c_0-1$. Hence $\col_{\yc{T}}(k-1) = \col_{\yc{T}}(k+m) = c_0-1$ and $\row_{\yc{T}} (k+m) > r_0$, as desired.
\end{proof}

We may now prove Proposition~\ref{prop:kmovesiffcondition1}. 

\begin{proof}[Proof of Proposition~\ref{prop:kmovesiffcondition1}]
 Set $(r_0,c_0) = \cell_{\yc{T}}(k)$.

\smallskip

\noindent $\Leftarrow$: As in Remark~\ref{YCTrows},  \ref{condition1} in Theorem~\ref{theorem.component change} implies that $\yc{T}$ has entries
\[
\TIKZ[scale =.8, font=\footnotesize]{
\draw
	(0,0) to ++(2,0)
	(0,1) to ++(2,0)
	(0,2) to ++(1,0)
	(0,0) to ++(0,2)
	(1,0) to ++(0,2)
	(2,0) to ++(0,1);
\draw[dotted] (1,2) to ++(1,0) to ++(0,-1);
\foreach \c [count=\x from 1] in {k \! - \! 1, k}
	{\node[] at (\x-.5, .5) {\strut$\c$};}
\foreach \c [count=\x from 1] in {k\! +\! m, *}
	{\node[] at (\x-.5, 1.5) {\strut$\c$};}
\node[left] at (0, 1.5){row $r_0+1$};
\node[left] at (0, .5){row $r_0$};
\node[above, inner sep=1pt] at (.5, 2) {\begin{tabular}{c}col\\[-3pt]$c_0 \sm 1$\end{tabular}};
\node[above, inner sep=1pt] at (1.5, 2) {\begin{tabular}{c}col\\[-3pt]$c_0$\end{tabular}};
}\]
which satisfy the conditions of Lemma~\ref{lemma:kmoves1}. Hence the Dyck pattern interval $I$ moves the cell of $k$.

\smallskip

\noindent $\Rightarrow$: Suppose that $I$ moves the cell of $k$. By Lemma~\ref{lemma:kmoves1}, this implies
    \[\col_{\yc{T}}(k-1) = \col_{\yc{T}}(k+m) = c_0-1
    \quad\text{and}\quad
    \row_{\yc{T}} (k+m) > r_0.
    \]
As above, let $r'= \row_{\yc{T}} (k+m)$. Further, the triple condition~\eqref{equation.triple} implies $\row_{\yc{T}}(k-1) \leqslant r_0$.  Let $x$ be the 
entry in cell $(r,c_0-2)$.  Since $k-1,x,k+m$ also satisfy the triple condition, we must have $k-1 < x < k+m$:
\[ \TIKZ[scale=.75]{
	\draw (0,2) rectangle +(2,1);
		\draw (1,2) to +(0,1);
		\node (k-1) at (0.5, 2.5) { $x$\strut};
		\node at (1.5, 2.5) {\footnotesize $k \! + \! m$\strut}; 
	\draw (1,0) rectangle +(1,1);
    		\node[] (k+m) at (1.5, .5) {\footnotesize $k \!-\! 1$\strut};
	\draw[dashed] (1.5, 1) to 
		+(-1, 1);
	\draw (1.5, 1) to node[fill=white, sloped, inner sep = .5pt]{\footnotesize$<$} +(0, 1);
	\node[right] at (2,2.5) {(row $r$)};
	\node[right] at (2,.5) {(row $r_0$).};
}\]
Since $I=[i,i+2m]$ is a Dyck pattern interval, (1) elements in $[k+1,i+m]$ occur in columns to the right of column $c_0$, and (2) elements in $[i+m+1,k+m-1]$ 
occur in strictly increasing columns to the left of column $c_0 - 1$ (see Lemma~\ref{lemma.column separation}). 
Hence $x=k+m-1$. And since $k+m-1$ cannot be paired with $k$ in the Dyck pattern
interval, we must also have $\col_{\yc{T}}(k-2) = c_0-2$:
\[ \TIKZ[scale=.75, xscale=1.5, font=\footnotesize]{
\draw (1.5,1) to ++(0,1);
\draw[myimplies, rounded corners] (1.5, 1.5) to ++(-.5,0) to ++(-.3, .5);
\draw[myimplies] (.5, 2) to  +(0, -1);
	\draw[fill=white] (0,2) rectangle +(2,1);
		\draw (1,2) to +(0,1);
		\node at (.5, 2.5) {\scriptsize $k \! + \! m\!-\! 1$\strut}; 
		\node at (1.5, 2.5) {\footnotesize $k  + m$\strut}; 
	\draw[fill=white] (0,0) rectangle +(2,1);
		\draw (1,0) to +(0,1);
    		\node[] (k-2) at (.5, .5) {\footnotesize $k -2$\strut};
    		\node[] (k-1) at (1.5, .5) {\footnotesize $k -1$\strut};
\node[above, inner sep=1pt] at (.5, 3) {\begin{tabular}{c}col\\[-3pt]$c_0 \sm 2$\end{tabular}};
\node[above, inner sep=1pt] at (1.5, 3) {\begin{tabular}{c}col\\[-3pt]$c_0 \sm 1$\end{tabular}};
	\node[right] at (2,2.5) {(row $r$)};
	\node[right] at (2,.5) {(row $r_0$).};
}\]
Continuing this line of reasoning implies that each of the first $c_0-1$ columns contains an 
entry from $[i+m+1,k+m]$ in row $r$ and an entry from $[i,k-1]$ in row $r_0$. This implies \ref{condition1} of Theorem \ref{theorem.component change}.
\end{proof}

\subsubsection{Shape change on smaller intervals}
\label{ss:smallerintervalshapechange}
Fix a Young composition tableau $\yc{T}$, Dyck pattern interval $I=[i,i+2m]$, and $k$ as in Theorem~\ref{theorem.component change}.
The next step is for us to study the shapes of the restricted skew tableaux $\yc{T}|_J$ and $I \cdot \yc{T}|_J$ for the interval $J = [1, i+2m]$. 
Our goal is to show that $\shape(\yc{T}|_J)$ and $\shape((I \cdot \yc{T})|_J)$ differ if and only if $I$ moves the cell of $k$ 
(Lemma~\ref{lemma:I on subtableau}~\eqref{lemmaI:2}). We describe how this shape changes in Lemma~\ref{lemma:I on subtableau} \eqref{lemmaI:3}.

\begin{lemma}
\label{lemma:I on subtableau}
Let $T\in \SYT(\lambda)$, $I = [i,i+2m]$ be a Dyck pattern interval, and $\cycle(I)=(k+m,k+m-1,\ldots,k)$. 
Let $\cell_{\yc{T}}(k) = (r_0,c_0)$ and set $J = [1,i+2m]$.
\begin{enumerate}
\item \label{lemmaI:rowlength} If $I$ moves the cell of $k$, then the length of the $(r_0+1)$-st row of $\yc{T}|_J$ is $c_0-1$.
\item  \label{lemmaI:3} If $I$ moves the cell of $k$,  
then $\shape((I\cdot \yc{T})|_J)$ is obtained by swapping rows $r_0$ and $r_0+1$ of $\shape(\yc{T}|_J)$.  \smallskip
\item \label{lemmaI:2}
We have that   
\[\shape(\yc{T}|_J) \neq \shape((I\cdot \yc{T})|_J)\] 
if and only if $I$ moves the cell of $k$.
\end{enumerate}
\end{lemma}

\begin{proof}
Recall that $\cycle(I) \cdot k = k+m$.

\medskip

\noindent \textit{Proof of parts \eqref{lemmaI:rowlength} and  \eqref{lemmaI:3}}: 
By assumption, $I$ moves the cell of $k$. By Proposition~\ref{prop:kmovesiffcondition1} we have that rows $r_0$ and $r_0+1$ are as described 
in Remark~\ref{YCTrows}. Since $\cycle(I) \cdot k = k+m$ and $\cycle(I) \cdot (k+m) = k+m-1$, it follows that in $I \cdot \yc{T}$, we have
$\cell_{I \cdot \yc{T}}(k+m) = (r_0+1,c_0)$. 

Define the following sets:
\begin{equation*}
    A_1:= [i,k] \qquad \text{and} \qquad B_1:= [i+m+1,k+m].
\end{equation*}Proposition~\ref{prop:kmovesiffcondition1} also implies that the letters in $A_1$ occur in the first 
$|A_1| = k-i +1$ columns of $\yc{T}$ with $k>i$ and the letters $B_1$ appear in the first $|A_1|-1 = k-i$ columns of $\yc{T}$. Furthermore,
$\yc{T}$ and $I \cdot \yc{T}$ are identical over the interval $[1,i-1]$ (a trivial case of Lemma~\ref{lemmaI:skew shape stability}). We next examine what 
happens to the remainder of $J$ in $I \cdot \yc{T}$ compared to $\yc{T}$.

By Lemma~\ref{lemma.column separation}, the elements in $[k,k+m]$ (which are the elements in $\cycle(I)$) appear in distinct columns of
$\yc{T}$. By the definition of $k$, the letter $k$ is the only element from $I$ in column $c_0$ of $\yc{T}$. Hence the letters in $[k+1,i+m]$ and
$[k+m+1,i+2m]$ appear in columns to the right of column $c_0$.  

For the following, see the left-hand side of Figure~\ref{fig:how I moves A2 and B2}: Let $c_1 \geqslant 0$ be maximal such that $c_1+i-1$ appears in 
column $c_1$ (i.e.\ the rightmost column where entries within $[k+1,i+m]$ 
appear in consecutive columns). Define 
\begin{equation*}
    A_2:= [k+1,c_1+i-1] \qquad \text{and} \qquad B_2:= [m+k+1,m+c_1+i-1],
\end{equation*}
and set $a_j:= k+j \in A_2$, $b_j = k+m+j \in B_2$ for $1 \leqslant j \leqslant c_1-c_0$.
By Remark~\ref{remark.outpacing}, since the entries in $A_2$ appear in consecutive columns between $c_0+1$ and $c_1$, the entries in $B_2$ must as well.
By the construction of $\varphi^{-1}$ in Definition~\ref{def.phi},
\[
	\row_{\yc{T}}(a_j)=r_0 \quad \text{and} \quad r_j:=\row_{\yc{T}}(b_j)<r_0
\]
since all entries above row $r_0$ are greater than $k$ by Remark~\ref{YCTrows} and the fact that the first column of $\yc{T}$ from bottom to top is increasing.  

Recall that we use the notation $x'=\cycle(I) \cdot x$ for every $x$ in $T$ (e.g. $k' = k+m$). Then
\[
	a_j'=a_j-1=k+j-1 \quad \text{and} \quad b_j'=b_j=k+m+j \quad \text{for $1\leqslant j \leqslant c_1-c_0$.}
\]
Since $\cell_{I \cdot \yc{T}}(k+m) = (r_0+1,c_0)$, we must have $\row_{I \cdot \yc{T}}(b_j')=r_0+1$ for $1\leqslant j \leqslant c_1-c_0$. Furthermore, 
$\row_{I \cdot \yc{T}}(a_j')=r_j$ since $r_j$ is the highest row of $I \cdot \yc{T}$ whose cell in column $c_0+j-1$ is less than $a_j'$. 
See Figure~\ref{fig:how I moves A2 and B2}. 

Next, set 
\begin{equation*}
    A_3 := [i,i+m] \setminus A_1 \cup A_2=[i+c_1,i+m],
\end{equation*}
as illustrated in the left-hand side of Figure~\ref{fig:how I moves I}.
By the definition of $c_1$, the entries in $A_3$ appear strictly to the right of column $c_1$. Hence by  the Column-Stability
Lemma~\ref{lem_dec_by_one}
\[
	\cell_{\yc{T}}(x) = \cell_{I\cdot \yc{T}}(x') \quad \text{for $x\in A_3$.}
\]
Finally, the letters in 
\[ [i,i+2m] \setminus A_1 \cup B_1 \cup A_2 \cup B_2 \cup A_3 = [k+m+c_1-c_0+1,i+2m]=[i+c_1+m,i+2m]\] also appear in columns 
$c_1+1$ or larger. Set $B_3$ to be the subset of elements in $[i+c_1+m,i+2m]$ appearing in row $r_0$ of $\yc{T}$, and $B_4$ to be those appearing 
in a row other than $r_0$:
\begin{align*}
B_3 := \{j \in [i+c_1+m,i+2m] ~|~ \row_{\yc{T}}(j) = r_0\},\\
B_4 := \{j \in [i+c_1+m,i+2m] ~|~ \row_{\yc{T}}(j) < r_0\}.
\end{align*}
Then it follows that for $x \in B_3$ and $y \in B_4$ 
\[ 
	\row_{I \cdot \yc{T}}(x) = r_0+1, \qquad \qquad \textrm{ and } \qquad \qquad \cell_{\yc{T}}(y) = \cell_{I \cdot \yc{T}}(y),
\]
since $x \in B_3$ is larger than any $z \in B_2$.  Since $k$ is the only entry in column $c_0$ of $\yc{T}$ from the interval $I$, row $r_0+1$ of $\yc{T}|_J$ 
has length $c_0-1=k-i$.  This is part \eqref{lemmaI:rowlength} of the lemma.

Since $k+m=k'$ is the only entry 
in $I \cdot \yc{T}$ in column $c_0$ from the interval $I$, no entry from $J$ will be placed in cell $(r_0,c_0)$ of $I \cdot \yc{T}$. So the length of row $r_0$ of 
$(I \cdot \yc{T})|_J$ equals the length of row $r_0+1$ of $\yc{T}|_J$. Hence $\shape((I\cdot \yc{T})|_J)$ is obtained by swapping rows $r_0$ and $r_0+1$ of 
$\shape(\yc{T}|_J)$. Again, see Figure~\ref{fig:how I moves I}.

\medskip

\noindent \textit{Proof of part \eqref{lemmaI:2}}:
The proof of \eqref{lemmaI:3} showed that if $I$ moves the cell of $k$, then $\shape(\yc{T}|_J) \neq \shape((I\cdot \yc{T})|_J)$. So we now assume that $I$ does not move the cell of $k$.

Let $a_1 < a_2 < \cdots < a_{d}$ be the non-empty entries in row $r_0$ of $\yc{T}|_J$ occurring to the right of column $c_0$.  If $a_1>k+m$, then 
column $c_0+1$ is stable under $I$ and the Column-Stability Lemma~\ref{lem_dec_by_one} implies that the remainder of the columns are stable under $I$.  

If $a_1 \in [k+1,k+m-1]$, we must have $a_1=k+1$ and there must be an entry $b_1>a_1$ in column $c_0+1$ with $b_1 \in I$ since $a_1$ is paired 
by Remark~\ref{remark.outpacing}. Let $r_1=\row_{\yc{T}}(b_1)$.  The triple rule implies that $r_1 < r_0$.  Then 
\[ 
	\row_{I \cdot \yc{T}}(a_1')=r_1 \qquad \qquad \textrm{ and } \qquad \qquad \row_{I \cdot \yc{T}}(b_1')=r_0. 
\]  
Repeating the argument, either $a_j>k+m$ and we can use Lemma~\ref{lem_dec_by_one} or $a_j=k+j$ and there is a $b_j>a_j$ in cell $(r_j, c_0+j)$
with $b_j\in I$ by Remark~\ref{remark.outpacing}. By the triple rule $r_j<r_0$ and
\[ 
	\row_{I \cdot \yc{T}}(a_j')=r_j \qquad \qquad \textrm{ and } \qquad \qquad \row_{I \cdot \yc{T}}(b_j')=r_0. 
\]   
The other entries in columns $c_0+1$ through $c_0+d$ appear in the same relative order, and so the lengths of the other rows do not change.  
Therefore the shape does not change and our proof is complete.  
\end{proof}

\begin{figure}
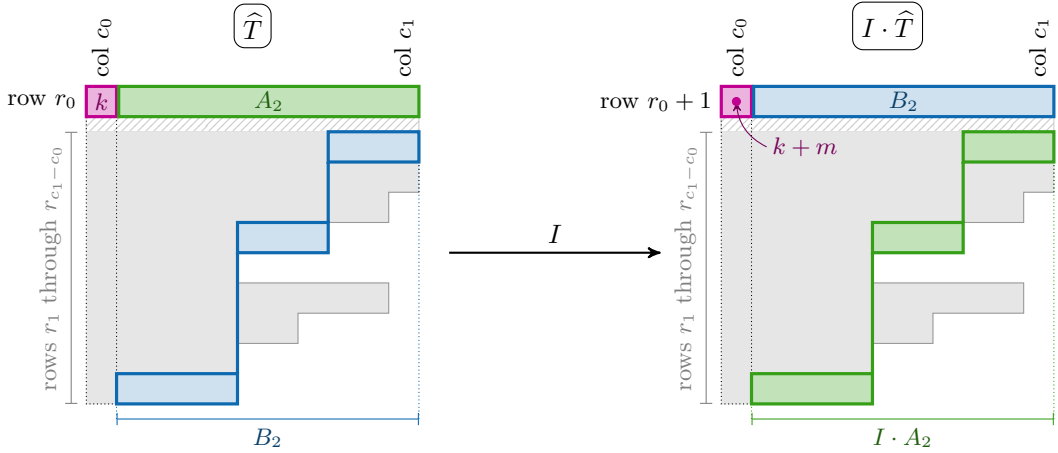

\[
\TIKZ[scale = .4, font=\small]{
\ABtwoCore
\draw[Bstrip, very thick, fill=Bstrip!20] 
	(x1) rectangle ++(4,1) to
	(x6) rectangle ++(3,1) to
	(x9) rectangle ++(3,1);
\draw [densely dotted, Bstrip] (1, -.5) to (1,0) (11, -.5) to (11,8);
\draw[|-|, Bstrip] (1, -.5) to node[below, Bstrip!60!black]{$B_2$} ++(10,0);
\draw[rose, very thick, fill=rose!30] (row r) rectangle coordinate (k) ++(1,1);
\path (row r) to ++(1,0) to ++(2pt,0)  coordinate (A2) to ++(10,1) to ++(-2pt,0) coordinate (endA2);
\draw[Astrip, very thick, fill=Astrip!30] 
	(A2) rectangle coordinate(A2label) (endA2);
\node[rose!60!black] at (k) {$k$};
\node[Astrip!60!black] at (A2label) {$A_2$};
\node[left] at (0,10) {row $r_0$};
\node[right, rotate=90] at (.5,10.5) {col $c_0$};
\node[right, rotate=90] at (10.5,10.5) {col $c_1$};
\node[draw, rounded corners] at (T) {\normalsize $\yc{T}$};
\draw[thick, ->] (12, 5) to node[above]{\normalsize $I$} ++(7,0);
\begin{scope}[shift={(21,0)}]
\ABtwoCore
\draw[Astrip, very thick, fill=Astrip!30] 
	(x1) rectangle ++(4,1) to
	(x6) rectangle ++(3,1) to
	(x9) rectangle ++(3,1);
\draw [densely dotted, Astrip] (1, -.5) to (1,0) (11, -.5) to (11,8);
\draw[|-|, Astrip] (1, -.5) to node[below, Astrip!60!black]{$I \cdot A_2$} ++(10,0);
\draw[rose, very thick, fill=rose!30] (row r) rectangle coordinate (k) ++(1,1);
\path (row r) to ++(1,0) to ++(2pt,0)  coordinate (A2) to ++(10,1) to ++(-2pt,0) coordinate (endA2);
\draw[Bstrip, very thick, fill=Bstrip!20] 
	(A2) rectangle coordinate(A2label) (endA2);
\draw[rose!60!black, <-, bend right] 
	(k) node[Bv, rose]{} to ++(1, -1.5) 
	node[right, rose!60!black, inner sep=2pt]{$k+m$};
\node[Bstrip!60!black] at (A2label) {$B_2$};
\node[left] at (0,10) {row $r_0+1$};
\node[right, rotate=90] at (.5,10.5) {col $c_0$};
\node[right, rotate=90] at (10.5,10.5) {col $c_1$};
\node[draw, rounded corners] at (T) {\normalsize $I \cdot \yc{T}\vphantom{T_x}$};
\end{scope}
}
\]
\caption{The local effect of $I$ on $A_2=[k+1,k+c_1-c_0]$ and $B_2=[k+m+1,k+m+c_1-c_0]$ in $\yc{T}|_{[1, i+2m]}$, focusing only on columns $c_0$ through 
$c_1$ and relevant rows: rows from $r_1=\row_{\yc{T}}(k+m+1)$ to $r_{c_1-c_0}=\row_{\yc{T}}(k+m+c_1-c_0)$, and either row $r_0$ (in $\yc{T}$) or $r_0+1$ 
(in $I \cdot \yc{T}$). Gray regions represent boxes in $\yc{T}|_{[1, i-1]}$ and $(I \cdot \yc{T})|_{[1, i-1]}$.}
\label{fig:how I moves A2 and B2}
\end{figure}

\begin{figure}
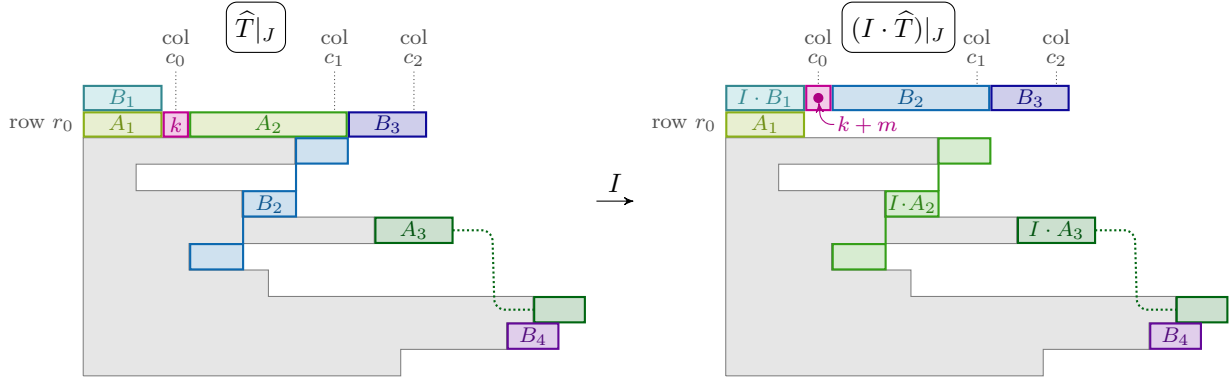

\[\TIKZ{
\node (T) at (0,0){\suffYT};
\node (IT) at (8.5,0){\suffIYT};
\node[draw, rounded corners] at (-0.5,2.3) {\normalsize $\yc{T}|_J$};
\node[draw, rounded corners] at (8,2.3) {\normalsize $(I \cdot \yc{T})|_J$};
\draw[->] (T) to node[above]{$I$} (IT);}\]
\caption{An illustration of the proof of Lemma \ref{lemma:I on subtableau} part \eqref{lemmaI:3}; on the left is $\yc{T}|_J$ where $J=[1,i+2m]$, $I = [i,i+2m]$,
 $I^-=[i, i+m] = \{k\} \cup \bigcup A_j$ and $I^+=[i+m+1, i+2m] = \bigcup B_j$; on the right is $(I\cdot \yc{T})|_J$.}\label{fig:how I moves I}
\end{figure}

\subsubsection{Proof of Theorem~\ref{theorem.component change}}
\label{ss:proofofcomponentchange}

We are at last ready to characterize when edges in the crystal skeleton change quasicrystal skeleton components
by proving Theorem \ref{theorem.component change}.  Our proof will use repeated applications of the triple condition~\eqref{equation.triple}.  
We remind the reader of the triple condition for standard Young tableaux from \eqref{equation.triple}: 
\[ \TIKZ[scale=.5]{
	\draw (0,2.5) rectangle +(2,1);
		\draw (1,2.5) to +(0,1);
		\node at (0.5, 3) {$x$};
		\node at (1.5, 3) {$y$}; 
	\draw (1,0) rectangle +(1,1);
    		\node at (1.5, .5) {$z$};  
	\draw (1.5, 1) to 
		 node[fill=white, sloped, inner sep = .5pt]{\footnotesize$<$} 
		+(-1, 1.5);
	\draw[densely dashed] (1.5, 1) to
        +(0, 1.5);
}
 \qquad\qquad 
 \text{If $z>x$, then $z>y$.}\]

From Lemma~\ref{lemma:I on subtableau}, we know that we cannot expect column stability, even outside of rows $r_0$ and $r_0+1$, since the entries from
$B_2$ in $\yc{T}$ are replaced by $I \cdot A_2$ in $I \cdot \yc{T}$ (see Figure~\ref{fig:how I moves I}). However, we can still describe how entries in $\yc{T}$ are changed by applying $I$ when \ref{condition1} is met.

\begin{lemma}\label{Gen_stability_lemma}
Assume \ref{condition1} is met.  The following are true for all columns $c$.
\begin{enumerate}[label=(\roman*), ref={(\roman*)}]
\item~\label{semistable1} Let $r \not= r_0, r_0+1$.
\begin{enumerate}
\item If $\yc{T}(r,c) \notin I$, then $\yc{T}(r,c) = (I \cdot \yc{T})(r,c)$.
\item If $\yc{T}(r,c) \in I$, then $(I \cdot \yc{T})(r,c) \in I$.
\end{enumerate}
\item~\label{semistable2} For $r = r_0, r_0+1$ we have
\[(I \cdot \yc{T})(r_0, c), (I \cdot \yc{T})(r_0+1, c) \in \{\yc{T}(r_0, c), \yc{T}(r_0+1, c)\} \cup I.\]
\end{enumerate}
\end{lemma}

Note that if a column is stable under $I$ (in the sense of Lemma~\ref{lem_dec_by_one}), then it satisfies both properties; in this sense, Lemma~\ref{Gen_stability_lemma} is a generalization of Lemma~\ref{lem_dec_by_one}. One can 
visualize Lemma~\ref{Gen_stability_lemma} as in Figure~\ref{fig:semistable diagram}.

\begin{figure}[t]
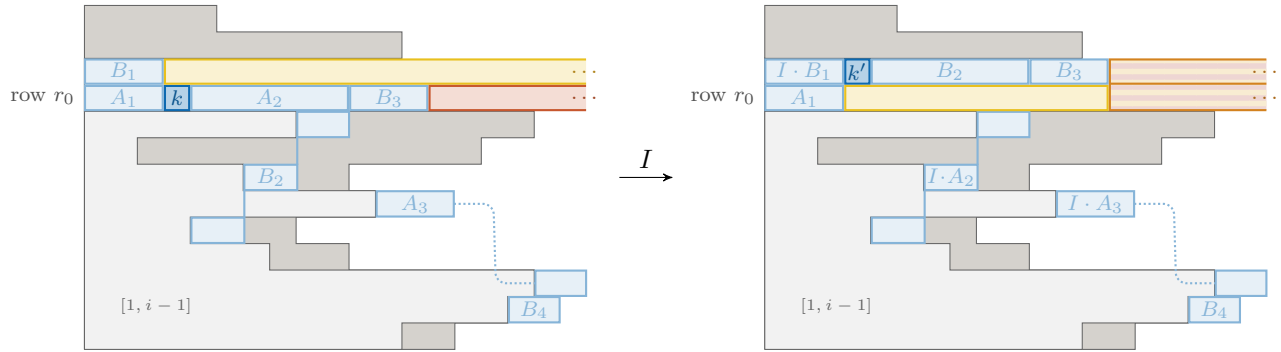

\[\TIKZ{
\node (T) at (0,0){\semistableYT};
\node (IT) at (9,0){\semistableIYT};
\draw[->] (T) to node[above]{$I$} (IT);}\]
\caption{Visualization of stability-like conditions stated in Lemma~\ref{Gen_stability_lemma}.}
\label{fig:semistable diagram}
\end{figure}

\begin{proof}
First note that since $i<k$, the leftmost column of $\yc{T}$ has $\yc{T}(r_0,1)=i$, $\yc{T}(r_0+1,1)=i+m+1$, and all other entries are outside of the interval
$I$.  Hence applying $I$ fixes all values in the leftmost column except the entry in row $r_0+1$, which becomes $i+m$. Therefore $(I \cdot \yc{T})(r_0+1,1)
=\yc{T}(r_0+1,1)-1$. This implies that the leftmost column of $\yc{T}$ is stable under $I$, which implies that it satisfies properties~\ref{semistable1} 
and~\ref{semistable2} since stability under $I$ is a special case of these properties.

Proceed by induction on the columns of $\yc{T}$, with the leftmost column as the base case.  Assume column $c-1$ satisfies the properties.  
For $y \in J=[1,i+2m]$ in column $c$, properties~\ref{semistable1} and \ref{semistable2} are satisfied by Lemma~\ref{lemma:I on subtableau}.  

Now fix $y > i+2m$ in column $c$ and assume all $z<y$ satisfy properties~\ref{semistable1} and \ref{semistable2}.  Note that $\cycle(I) \cdot y = y$, 
and let $r_y := \row_{I\cdot \yc{T}}(y)$. 
Our goal is to show that either $\row_{\yc{T}}(y) = r_y$, or  both $r_y$ and  $\row_{\yc{T}}(y)$ are in $\{r_0, r_0+1\}$. 

\def\GoodRowsT{\mathcal{R}}
\def\GoodRowsIT{\mathcal{R}'}

Since the rows of a standard Young composition tableau must increase when read left to right, the entry immediately to the left of $y$ must be less 
than $y$.  Let $\GoodRowsT$ (resp. $\GoodRowsIT$) be the set of rows of $\yc{T}$ (resp. $I \cdot \yc{T}$) containing an entry in column $c-1$ 
which is less than $y$; that is, 
\[
	\GoodRowsT := \left\{ r ~\left|~ \yc{T}(r, c-1) < y\right.\right\} 
	\quad \text{ and }\quad 
	\GoodRowsIT := \left\{ r' ~\left|~ (I \cdot \yc{T})(r', c-1) < y\right.\right\}.
\]  
By the assumptions placed on column $c-1$, we have 
\[
	\GoodRowsT \setminus \{r_0, r_0+1\} = \GoodRowsIT \setminus \{r_0, r_0+1\}
 	\quad \text{and} \quad |\GoodRowsT \cap \{r_0, r_0+1\}| = |\GoodRowsIT \cap \{r_0, r_0+1\}|.
\]
Namely, either $\GoodRowsT = \GoodRowsIT$ or they differ only in rows $r_0$ and $r_0+1$.  If they differ in rows $r_0$ and $r_0+1$, then 
$|\GoodRowsT \cap \{r_0, r_0+1\}| = |\GoodRowsIT \cap \{r_0, r_0+1\}|$ so that  either 
\begin{itemize}
    \item $r_0 \in \GoodRowsT$ and $r_0+1 \in \GoodRowsIT$, or 
    \item $r_0+1 \in \GoodRowsT$ and $r_0 \in \GoodRowsIT$.
\end{itemize}

First assume $r_y \notin \{r_0, r_0+1\}$ and let $r_1 := \row_{\yc{T}}(y)$.  If $r_1 < r_y$, then in $\yc{T}$ by the triple rule, we have 
\[
\TIKZ[scale=.5]{
\Triple{x}{z}{y}
\begin{scope}[font=\footnotesize, color=black!70]
\node[right] at (top row) {row $r_y$\strut};
\node[right] at (bot row) {row $r_1$\strut};
\node[above, inner sep=1pt] at (1.5, \rowR+1) {\begin{tabular}{c}col\\[-3pt]$c$\end{tabular}};
\end{scope}
}
\]
and the entry $z := (I \cdot \yc{T})(r_y,c)$ is less than $y$.  Hence $z$ was inductively placed under $\varphi^{-1}$ as expected.  This implies that either $z \in I$ 
(implying $(I \cdot \yc{T})(r_y,c) \in I$) or $z \notin I$ in which case $(I \cdot \yc{T})(r_y,c) = z$.  Either way, this contradicts our assumption that 
$(I \cdot \yc{T})(r_y,c) =y$ since $y \notin I$. Therefore $r_1 \geqslant r_y$.  If $r_1=r_y$ we are done. Hence we may assume that $r_1>r_y$.

For all $r' \in \GoodRowsIT$ with $r'>r_y$, in $I \cdot \yc{T}$ by the triple rule
\[\TIKZ[scale=.5]{
\Triple{x}{z}{y}
\begin{scope}[font=\footnotesize, color=black!70]
\node[right] at (top row) {row $r'$\strut};
\node[right] at (bot row) {row $r_y$\strut};
\node[above, inner sep=1pt] at (1.5, \rowR+1) {\begin{tabular}{c}col\\[-3pt]$c$\end{tabular}};
\end{scope}
}\]
the entry  $z:= (I \cdot \yc{T})(r', c)$ is less than $y$. Hence $z$ was inductively placed under $\varphi^{-1}$ as expected, meaning that we have one of the 
following cases (aligning with properties \ref{semistable1} and \ref{semistable2} above, respectively): 
\begin{enumerate}[(i')]
\item If $r' \ne r_0, r_0+1$, then either $z\notin I$ or $z \in I$.  If $z \notin I$, then $\yc{T}(r',c)=z$. If $z \in I$, then $\yc{T}(r',c) \in I$.  Either way, $\yc{T}(r',c) \not= y$ and so this case cannot happen.  

\item If $r' \in \{r_0, r_0+1\}$ and $z \notin I$, then either $\yc{T}(r_0,c)=z$ or $\yc{T}(r_0+1,c)=z$.  If $r' \in \{r_0, r_0+1\}$ and $z \in I$, then 
$\yc{T}(r_0,c) \in I$ or $\yc{T}(r_0+1,c) \in I$.
\end{enumerate}
Therefore, since by assumption $r_1=\row_{\yc{T}}(y) >r_y$, we have $\row_{\yc{T}}(y) \in \{ r_0, r_0+1 \}$.  Since $y \notin I$, the sets of values in 
column $c$ of $\yc{T}$ and of $I \cdot \yc{T}$ that are less than $y$ are in bijection via 
\[
	\{ z ~|~ \col_{\yc{T}}(z) = c  \text{ and }  z < y\} \xrightarrow{~\cycle(I)} 
	\{ z' ~|~ \col_{I \cdot \yc{T}}(z') = c  \text{ and }  z' < y\}.
\] 
Hence, for all $r' \in \GoodRowsT \setminus \{r_0, r_0+1\} =  \GoodRowsIT \setminus \{r_0, r_0+1\}$ with $r' > r_y$, we have  
\[
	\yc{T}(r',c)<y \quad \text{ if and only if } \quad (I \cdot \yc{T})(r',c)<y.
\]
Furthermore the number of cells in rows $r_0$ and $r_0+1$ of a fixed column $c$ containing values less than $y$ is the same in both $\yc{T}$ and in 
$I \cdot \yc{T}$.  Hence if $\row_{\yc{T}}(y) \in \{r_0, r_0+1 \}$ we must have $\row_{I \cdot \yc{T}}(y) \in \{r_0, r_0+1 \}$ for otherwise we would have 
too many entries above row $r_y$ in column $c-1$ that are less than $y$.  Therefore $y$ is either in row $r_y$ of $\yc{T}$, or moves between rows 
$r_0$ and $r_0+1$ under $I$. 
\end{proof}

\begin{proof}[Proof of Theorem~\ref{theorem.component change}]
Set $\alpha:=\shape(\yc{T})$ and $\beta:=\shape(I \cdot \yc{T})$.

First note that if \ref{condition1} is not met, then by Proposition~\ref{prop:kmovesiffcondition1} and Lemma~\ref{lemma:I on subtableau} \eqref{lemmaI:2}, 
$\shape(\yc{T}|_J) = \shape(I \cdot \yc{T}|_J)$ for $J = [1, i+2m]$. Thus by Lemma~\ref{lemmaI:skew shape stability}, $\alpha = \beta$.

Now suppose that \ref{condition1} is met, so that by Proposition~\ref{prop:kmovesiffcondition1}, we know $I$ moves the cell of $k$. It remains to 
show that $\alpha \neq \beta$ if and only if \ref{condition2} is met.  To that end, consider what happens when we build $I \cdot \yc{T}$ 
column-by-column under the map $\varphi^{-1}$.

Lemma~\ref{Gen_stability_lemma} implies that all columns of $\yc{T}$ satisfy properties \ref{semistable1} and \ref{semistable2}. Thus, the remaining work is 
to describe the behavior under $I$ in rows $r_0$ and $r_0+1$.  Recall that $\cell_{\yc{T}}(k) = (r_0,c_0)$. In columns $c< c_0$, we have column stability from 
Lemma~\ref{lem_dec_by_one}. For columns $c \geqslant c_0$, parts \eqref{lemmaI:rowlength} and \eqref{lemmaI:3} of Lemma~\ref{lemma:I on subtableau} 
imply that 
\begin{itemize}
    \item $\yc{T}(r_0+1,c) \notin I$, and
    \item  $\yc{T}(r_0,c)  \in I$ if and only if $(I \cdot \yc{T})(r_0+1,c)  \in I$.
\end{itemize}
Hence, for $c \geqslant c_0$, by property \ref{semistable2} the entries in rows $r_0$ and $r_0+1$ either remain fixed under $I$ or the entry in row 
$r_0+1$ moves to row $r_0$ and either the entry from row $r_0$ or an entry from $I$ is placed into row $r_0+1$.

When $c=c_0$, we specifically have in $\yc{T}$ and $I \cdot \yc{T}$ respectively
\begin{equation}
\label{equation.col c0}
\TIKZ[xscale=.75, yscale=.6, font=\small]{
	\draw (0,1) rectangle +(2,2) 
		(1,1) to ++(0,2) (0,2) to ++(2,0);
		\node (x1) at (.5, 1+.5) {\footnotesize $k\! -\!1$\strut};
		\node (x2) at (.5, 2+.5) {\footnotesize $k\!+\!m$\strut};
		\node (y1) at (1.5, 1+.5) {$k$\strut};
    		\node (y2) at (1.5, 2+.5) {$y$\strut};  
	\begin{scope}[color=black!70]
		\begin{scope}[font=\footnotesize]
			\node[right] at (2,2+.5) {row $r_0+1$\strut};
			\node[right] at (2,1+.5) {row $r_0$\strut};
		\node[above, inner sep=1pt] at (1.5, 2+1) {\begin{tabular}{c}col\\[-3pt]$c_0$\end{tabular}};
	\end{scope}
	\end{scope}
	\draw[->] (4,2) to node[above]{$I$} ++(1.5,0);
\begin{scope}[shift={(7,0)}]
	\draw (0,1) rectangle +(2,2) 
		(1,1) to ++(0,2) (0,2) to ++(2,0);
		\node (x1') at (.5, 1+.5) {\footnotesize $k\!-\!1$\strut};
		\node [Bv] (x2') at (.5, 2+.5) {};
		\node (y1') at (1.5, 1+.5) {$y$\strut};
    		\node (y2') at (1.5, 2+.5) {\footnotesize $k\!+\!m$\strut};  
	\draw[<-, rounded corners] (x2') to ++(0,1) to ++(-.4,0) node[left] {\footnotesize$k\!+\!m\!-\!1$};
	\begin{scope}[color=black!70]
		\begin{scope}[font=\footnotesize]
			\node[right] at (2,2+.5) {row $r_0+1$\strut};
			\node[right] at (2,1+.5) {row $r_0$\strut};
		\node[above, inner sep=1pt] at (1.5, 2+1) {\begin{tabular}{c}col\\[-3pt]$c_0$\end{tabular}};
	\end{scope}
	\end{scope}
\end{scope}	
}
\end{equation}
where $y > i+2m$ (including if $y$ is empty, in which case it has value infinity by convention).

For the remaining columns, we claim that 
\[ 
	(I \cdot \yc{T})(r_0+1,c) < (I \cdot \yc{T})(r_0,c) \textrm{  for all }c \geqslant c_0 . 
\]
To see this, first note that if 
\[ 
	\yc{T}(r+1,c')< \yc{T}(r,c') \textrm{ and  }\yc{T}(r+1,c'+1)> \yc{T}(r,c'+1), \textrm{ then }\yc{T}(r,c'+1) > \yc{T}(r+1,c'),
\]
and so cells $(r+1,c'), (r,c'+1),$ and $(r+1,c'+1)$ violate the triple condition~\eqref{equation.triple}.  Therefore in any Young composition tableau $\yc{T}$ 
and any row $r$, 
\[ 
	\textrm{ if } \yc{T}(r+1,c)< \yc{T}(r,c), \textrm{ then }\yc{T}(r+1,c')< \yc{T}(r,c') \textrm{ for all }c' \geqslant c. 
\]
Since $(I \cdot \yc{T})(r_0+1,c_0) < (I \cdot \yc{T})(r_0,c_0)$, we have $(I \cdot \yc{T})(r_0+1,c) < (I \cdot \yc{T})(r_0,c)$ for all $c \geqslant c_0$.

We are now ready to apply induction across the columns of $\yc{T}$. Recall the convention that for any empty cell $(r, c)$ in $\yc{T}$ and any $x \in [n]$, 
we set $x < \yc{T}(r, c)$. Lemma~\ref{Gen_stability_lemma} property~\ref{semistable1} implies that the lengths of all rows $r \ne r_0, r_0+1$ are the same 
in $\yc{T}$ and $I \cdot \yc{T}$. Recall columns $c_1$ and $c_2$ are defined in the proof of Lemma~\ref{lemma:I on subtableau} and 
Figure~\ref{fig:how I moves I}. Let $c_3$ be the rightmost column where $\yc{T}(r_0, c) < \yc{T}(r_0+1, c)$; thus $c_3 \geqslant c_2$. Then 
Lemma~\ref{Gen_stability_lemma}, Lemma~\ref{lemma:I on subtableau}, the observation in~\eqref{equation.col c0} about what happens to column $c_0$, 
and the fact that $(I \cdot \yc{T})(r_0+1,c) < (I \cdot \yc{T})(r_0,c)$ for all $c \geqslant c_0$ come together to say that non-empty boxes in rows $r_0$ and 
$r_0+1$ will effectively swap rows from columns $c_0$ through $c_3$, and will be fixed thereafter. See Figure~\ref{figure.DEchange} and 
Remark~\ref{remark.DEchange}.

In particular, the shape will change under $I$ if and only if an empty box changes place with a non-empty box.  This will happen if and only if 
row $r_0+1$ of $\yc{T}$ has length less than $ c_3$ (so that row $r_0$ of $\yc{T}$ has length precisely $c_3$ by definition). In other words, $\alpha \ne \beta$ 
if and only if \ref{condition2} is met.
\end{proof}

\begin{remark}\label{remark.DEchange}
We add visual intuition in Figure~\ref{figure.DEchange} for the proof of Theorem~\ref{theorem.component change} by building out from 
Figure~\ref{fig:how I moves I}. Let
\[
	D_1 := \{ \yc{T}(r_0+1, c) ~|~ c_0 \leqslant c \leqslant c_2\}
\]
be the values in row $r_0+1$ of $\yc{T}$ that replace elements of $I$ in row $r_0$ under the action of $\cycle(I)$; let 
\begin{equation*}
	E_1 := \{ \yc{T}(r_0, c) ~|~ c_2+1 \leqslant c \leqslant c_3\} \qquad \text{and} \qquad
	D_2 := \{ \yc{T}(r_0+1, c) ~|~ c_2+1 \leqslant c \leqslant c_3\}
\end{equation*}
be the remaining values in these two rows such that the value in row $r_0$ is smaller (and non-empty); and let 
\begin{equation*}
E_2 := \{ \yc{T}(r_0, c) ~|~ c_3 < c\} \qquad  \text{ and } \qquad
D_3 := \{ \yc{T}(r_0+1, c) ~|~ c_3 < c\}
\end{equation*}
be all remaining values in these two rows (where the value in row $r_0$ is larger or both are empty). 

As pictured, $D_1$ contains only values greater than $i+2m$, and hence moves under $I$ since 
\[ (I \cdot \yc{T})(r_0+1,c) < (I \cdot \yc{T})(r_0,c) \textrm{ for all }
c \geqslant c_0. \] Moreover, regions $D_2$ and $E_1$ are cells where $i+2m < \yc{T}(r_0,c) < \yc{T}(r_0+1,c)$, so also move for this same reason. 
Regions $D_3$ and $E_2$ are cells where $i+2m < \yc{T}(r_0+1,c) \leqslant \yc{T}(r_0,c)$, 
and hence do not move under $I$ since they already have the smaller entry in the higher row.  Once we have passed column $c_3$, the 
values in row $r_0$ of $\yc{T}$ will continue to exceed those in row $r_0+1$ for all remaining columns. 

Hence, \ref{condition2} is met if and only if there is at least one empty cell in regions $D_1$ or $D_2$. And if so, then regions $E_2$ and $D_3$ are 
also empty.
\end{remark}

\begin{figure}[!h]
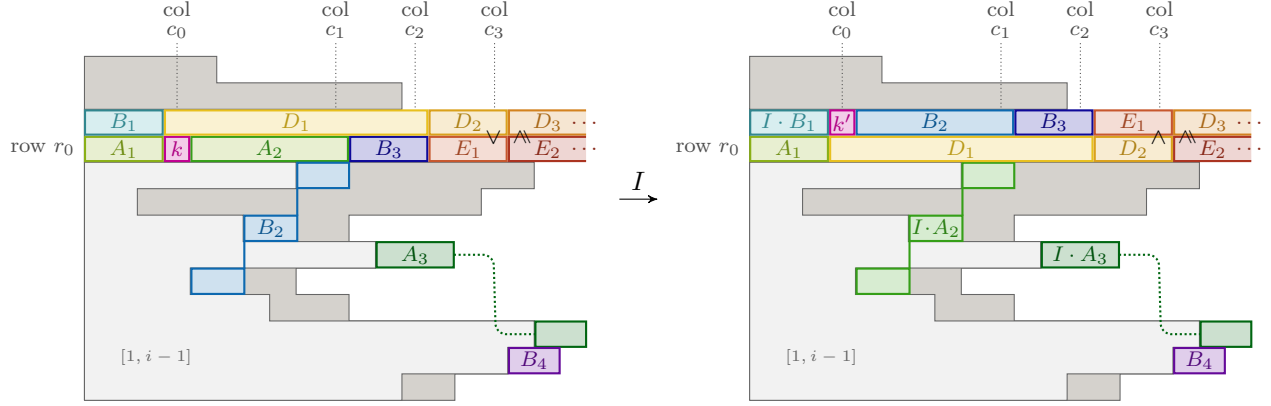

\[\TIKZ{
\node (T) at (0,0){\suffYTallColor};
\node (IT) at (8.8,0){\suffIYTallColor};
\draw[->] (T) to node[above]{$I$} (IT);}\]
\caption{An illustration of the proof of Theorem~\ref{theorem.component change}, detailing how values move under the action of $I$ given that $I$ and 
$\yc{T}$ satisfy \ref{condition1}. See also Remark~\ref{remark.DEchange} and Figure~\ref{fig:how I moves I}. 
}\label{figure.DEchange}
\end{figure}
\medskip
\begin{example}
Consider the following example where $n=150$ and $I = [84, 112]$ (so that $i = 84$ and $m=14$) and $k = 87$. 

\newcommand\ShapeChangeExFill[3]
{
	\filldraw[#3!15] (#1-1,#2-1) rectangle ++(1,1);
	\ifnum\thes>99
	\node[#3!70!black] at (#1-.5, #2-.5){\tiny \strut$\thes$}; 
	\fi
	\ifnum\thes<100
	\node[#3!70!black] at (#1-.5, #2-.5){\strut $\thes$}; 
	\fi
	\stepcounter{s}
}
\[\TIKZ[font=\scriptsize, scale=.415]{
\foreach \x [count=\c from 0] in {4, 10, 13, 15}{
\draw[densely dotted] (\x -.5,11) to (\x-.5, 13.5) node[above] {$c_\c$};}
\setcounter{s}{1}
\foreach \L [count=\r from 1] in {12,16,17,7,4,11,6,2,8}{
\foreach \c in {1, ..., \L}{
	\filldraw[black!5] (\c-1,\r-1) rectangle ++(1,1);
	\node[black!70] at (\c-.5, \r-.5){\scriptsize \strut $\thes$}; 
	\stepcounter{s}
}
}
\foreach \c in {1,2,3}{	\ShapeChangeExFill{\c}{10}{A1strip}}
\ShapeChangeExFill{4}{10}{kstrip}
\foreach \c in {5,...,10}{\ShapeChangeExFill{\c}{10}{A2strip}}
\foreach \c in {12,13,14}{\ShapeChangeExFill{\c}{6}{A3strip}}
\foreach \c in {18,19}{\ShapeChangeExFill{\c}{3}{A3strip}}
\foreach \c in {1,2,3}{\ShapeChangeExFill{\c}{11}{B1strip}}
\foreach \r [count=\c from 5] in {5,5,7,7,9,9}{\ShapeChangeExFill{\c}{\r}{B2strip}}
\foreach \c in {11,12,13}{\ShapeChangeExFill{\c}{10}{B3strip}}
\foreach \c in {17,18}{\ShapeChangeExFill{\c}{2}{B4strip}}
\ShapeChangeExFill{1}{12}{black}
\foreach \c in {11,12}{\ShapeChangeExFill{\c}{9}{black}}
\ShapeChangeExFill{3}{8}{black}
\foreach \c in {4, ..., 7}{\ShapeChangeExFill{\c}{11}{D1strip}}
\foreach \c in {2,...,6}{\ShapeChangeExFill{\c}{12}{black}}
\foreach \c in {1,...,3}{\ShapeChangeExFill{\c}{13}{black}}
\foreach \c in {14,15}{\ShapeChangeExFill{\c}{10}{E1strip}}
\foreach \c in {8, ..., 13}{\ShapeChangeExFill{\c}{11}{D1strip}}
\foreach \c in {9,10}{\ShapeChangeExFill{\c}{7}{black}}
\foreach \c in {14,15}{\ShapeChangeExFill{\c}{11}{D2strip}}
\foreach \c in {16,17,18}{\ShapeChangeExFill{\c}{11}{D3strip}}
\foreach \c in {7,8,9}{\ShapeChangeExFill{\c}{12}{black}}
\foreach \c in {16,17}{\ShapeChangeExFill{\c}{10}{E2strip}}
\foreach \c in {4,5}{\ShapeChangeExFill{\c}{13}{black}}
\begin{scope}
\Comp{12,18,19,7,6,14,10,3,12,17,18, 9, 5}
\end{scope}
\draw[very thick] (0,0)  
\foreach \x [count = \y from 1] in {12,18,19,7,6,14,10,3,12,15,14, 6, 3}{
	to (\x, \y-1) to (\x, \y)
}
to (0,13) to (0,0);
} 
\xrightarrow{\ I\ }
\TIKZ[font=\scriptsize, scale=.415]{
\foreach \x [count=\c from 0] in {4, 10, 13, 15}{
\draw[densely dotted] (\x -.5,11) to (\x-.5, 13.5) node[above] {$c_\c$};}
\setcounter{s}{1}
\foreach \L [count=\r from 1] in {12,16,17,7,4,11,6,2,8}{
\foreach \c in {1, ..., \L}{
	\filldraw[black!5] (\c-1,\r-1) rectangle ++(1,1);
	\node[black!70] at (\c-.5, \r-.5){\scriptsize \strut $\thes$}; 
	\stepcounter{s}
}
}
\foreach \c in {1,2,3}{	\ShapeChangeExFill{\c}{10}{A1strip}}
\foreach \r [count=\c from 5] in {5,5,7,7,9,9}{\ShapeChangeExFill{\c}{\r}{A2strip}}
\foreach \c in {12,13,14}{\ShapeChangeExFill{\c}{6}{A3strip}}
\foreach \c in {18,19}{\ShapeChangeExFill{\c}{3}{A3strip}}
\foreach \c in {1,2,3}{\ShapeChangeExFill{\c}{11}{B1strip}}
\ShapeChangeExFill{4}{11}{kstrip}
\foreach \c in {5,...,10}{\ShapeChangeExFill{\c}{11}{B2strip}}
\foreach \c in {11,12,13}{\ShapeChangeExFill{\c}{11}{B3strip}}
\foreach \c in {17,18}{\ShapeChangeExFill{\c}{2}{B4strip}}
\ShapeChangeExFill{1}{12}{black}
\foreach \c in {11,12}{\ShapeChangeExFill{\c}{9}{black}}
\ShapeChangeExFill{3}{8}{black}
\foreach \c in {4, ..., 7}{\ShapeChangeExFill{\c}{10}{D1strip}}
\foreach \c in {2,...,6}{\ShapeChangeExFill{\c}{12}{black}}
\foreach \c in {1,...,3}{\ShapeChangeExFill{\c}{13}{black}}
\foreach \c in {14,15}{\ShapeChangeExFill{\c}{11}{E1strip}}
\foreach \c in {8, ..., 13}{\ShapeChangeExFill{\c}{10}{D1strip}}
\foreach \c in {9,10}{\ShapeChangeExFill{\c}{7}{black}}
\foreach \c in {14,15}{\ShapeChangeExFill{\c}{10}{D2strip}}
\foreach \c in {16,17,18}{\ShapeChangeExFill{\c}{11}{D3strip}}
\foreach \c in {7,8,9}{\ShapeChangeExFill{\c}{12}{black}}
\foreach \c in {16,17}{\ShapeChangeExFill{\c}{10}{E2strip}}
\foreach \c in {4,5}{\ShapeChangeExFill{\c}{13}{black}}
\begin{scope}
\Comp{12,18,19,7,6,14,10,3,12,17,18, 9, 5}
\end{scope}
\draw[very thick] (0,0)  
\foreach \x [count = \y from 1] in {12,18,19,7,6,14,10,3,12,14,15, 6, 3}{
	to (\x, \y-1) to (\x, \y)
}
to (0,13) to (0,0);
}
\]
This example meets \ref{condition1}, with $r_0=10$ and $c_0, \dots, c_3$ as marked. However, it does not meet \ref{condition2}: 
row 11 is longer than row 10 in $\yc{T}$, due to 
\[E_2 = \{147, 148\} \quad \text{ and } \quad D_3 = \{141, 142, 143\}\]
being non-empty. However, if we restrict to the interval $J = [1, 139]$ (the thickly outlined region), then \ref{condition2} is met, and we change shape by swapping rows 10 and 11. The same can be said for a restriction to $J = [1, N]$ for any $112< N < 140$.
\end{example}

\subsection{Connectedness of the quasicrystal skeleton}
\label{section.connectedness}

Note that, as explored in~\cite{MNS.2025}, the subgraph of the crystal  corresponding to terms contributing to a (Young) quasisymmetric Schur 
function is not necessarily connected, nor is the subgraph corresponding to a Demazure atom.  Therefore, unlike in the quasicrystal case (as stated 
in Theorem~\ref{theorem.QC connected}), one would not expect the quasicrystal skeleton to be connected in general.  For example 
$\QCS_{3,1,3,1}$ and $\QCS_{4,2,1,2}$ are disconnected, see for example Figure~\ref{figure.QCS3131}.

\begin{figure}[t]
\centering
\includegraphics{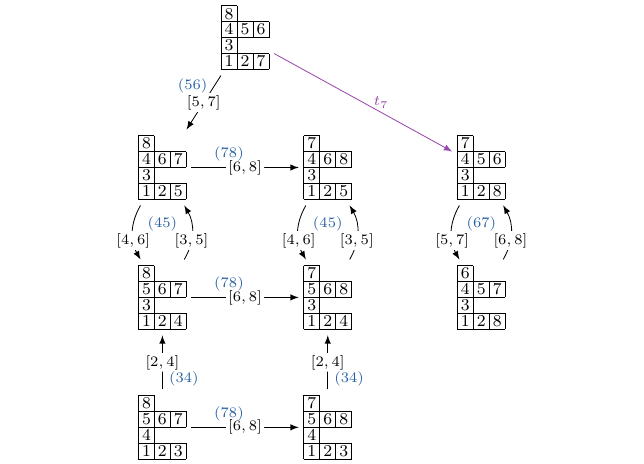}
\caption{The augmented quasicrystal skeleton $\QCS^+_{3,1,3,1}$ (not displaying $t_i$ edges which replicate existing crystal skeleton edges).
Without the displayed $t_7$ arrow, $\QCS_{3,1,3,1}$ is disconnected.
\label{figure.QCS3131}}
\end{figure}

We therefore need to introduce additional edges in the quasicrystal skeleton in order to make it connected.  
\begin{definition}\label{def:Youngcompositionentryswap}
Let $i$ and $i+1$ be two entries in a Young composition tableau $\yc{T}$ such that 
\[ |\col_{\yc{T}}(i+1)-\col_{\yc{T}}(i)|>1 \]
with $i+1$ in a row higher than $i$. Define $t_i(\yc{T})$ to be the Young composition tableau obtained by swapping the entries $i$ and $i+1$.
\end{definition}
Note that if a Young composition tableau satisfies Definition~\ref{def:Youngcompositionentryswap},  the entries must satisfy this same condition in the corresponding standard Young tableau $T=\varphi(\yc{T})$.

\begin{lemma}
The diagram $t_i(\yc{T})$ is a Young composition tableau.
\end{lemma}

\begin{proof}
Condition (1) of Definition~\ref{def:SSYCT} is satisfied since $i$ and $i+1$ cannot be in the same row.  Condition (2) is satisfied since $i$ and $i+1$ 
cannot both be in the leftmost column.  We must check the third condition.  We only need to check triples involving $i$ and $i+1$.  But since 
$|\col_{\yc{T}}(i)-\col_{\yc{T}}(i+1)|>1$, no triple involves both $i$ and $i+1$.  Any triple involving only one of $i$ and $i+1$ will 
maintain the same relative orders when $i$ and $i+1$ are swapped.  Therefore all triples will continue to satisfy condition (3) in $t_i(\yc{T})$ and 
therefore $t_i(\yc{T})$ is indeed a Young composition tableau.
\end{proof}

We next introduce a method of introducing edges to the quasicrystal skeleton so that it is connected. 

\begin{definition}
We define the \defn{augmented quasicrystal skeleton} $\QCS^+_{\alpha}$ to be the quasicrystal skeleton $\QCS_{\alpha}$ together with all edges 
of the form 
\begin{equation}
\label{equation.transposition edge}
	\yc{T} \stackrel{t_i}{\longrightarrow} t_i(\yc{T})
\end{equation}
for $\yc{T} \in \QCS_\alpha$.
\end{definition}

\begin{theorem}
\label{theorem.connectivity}
The augmented quasicrystal skeleton $\QCS^+_{\alpha}$ is connected for all compositions $\alpha$.
\end{theorem}

\begin{proof}
We prove the claim by induction on the size of $\alpha$.  The base case is trivial.  Assume $\QCS^+_{\delta}$ is connected for all $\delta$ such that 
$| \delta| < n$ and let $\alpha=(\alpha_1,\alpha_2, \hdots , \alpha_{\ell})$ be a composition of $n$.   First note that removing the cell containing $n$ 
leaves a valid Young composition tableau, since the removal of $n$ is akin to replacing $n$ with infinity.  Let $\gamma$ be the shape of a diagram 
that can be obtained from a Young composition tableau of shape $\alpha$ by removing the cell containing $n$.  We claim that the edges in 
$\QCS^+_{\gamma}$ are a subset of the edges in $\QCS^+_{\alpha}$.  To see this, first let $e$ be an edge in $\QCS^+_{\gamma}$.  If it comes 
from a Dyck pattern interval, that Dyck pattern interval remains a Dyck pattern interval when $n$ is inserted into the reading word.  Hence it is an 
edge in $\QCS^+_{\alpha}$.  If $e$ is an edge of the form~\eqref{equation.transposition edge} with $1\leqslant i \leqslant n-2$ for some Young 
composition tableau $\yc{T}$ of shape $\gamma$, then $i$ and $i+1$ remain in the same positions when $n$ is inserted and thus $t_i$ still applies 
to the corresponding Young composition tableau of shape $\alpha$. So the edges of $\QCS^+_{\gamma}$ are contained in the edges of $\QCS^+_{\alpha}$.  

Note that there is always a Young composition tableau of shape $\alpha$ with $n$ at the end of the last row of $\alpha$, namely the superstandard 
tableau.  Let $\beta$ be the composition obtained by decreasing the last part ($\alpha_{\ell}$) of $\alpha$ by one.  (If $\alpha_{\ell}=1$ then simply 
delete the last part of $\alpha$.)  Then $\QCS^+_{\beta}$ is connected by the inductive hypothesis, as is $\QCS^+_{\gamma}$.   

Let $d_{\beta}$ be the cell removed from $\alpha$ to construct $\beta$ and let  $d_{\gamma}$ be the cell removed from one of the first $\ell-1$ 
rows of $\alpha$ to construct $\gamma$.  It suffices to show that there exists a Young composition tableau $\yc{T}_{\beta} \in \QCS^+_{\beta}$ 
and a Young composition tableau $\yc{T}_{\gamma} \in \QCS^+_{\gamma}$ such that there is an edge in $\QCS^+_\alpha$ between 
$\widetilde{T}_{\beta}$ and $\widetilde{T}_{\gamma}$, where $\widetilde{T}_{\beta}$ and $\widetilde{T}_{\gamma}$ are the Young composition
tableaux of shape $\alpha$ constructed by appending an $n$ to $\yc{T}_{\beta}$ in position $d_{\beta}$ and to $\yc{T}_{\gamma}$ in position 
$d_{\gamma}$ respectively. 

First assume that $\alpha_{\ell} \geqslant 2$.  Let $\yc{T}_{\gamma}$ be the superstandard Young composition tableau of shape $\gamma$.  Then 
$\cell_{ \widetilde{T}_{\gamma}}(n-1) = ( \ell, c)$ with $c=\alpha_\ell \geqslant 2$.  So $\cell_{ \widetilde{T}_{\gamma}}(n-2) = ( \ell, c-1)$.
If $\col_{\widetilde{T}_{\gamma}}(n)<c-1$ or $\col_{\widetilde{T}_{\gamma}}(n)>c+1$ then $t_{n-1}$ is an edge between $\widetilde{T}_{\gamma}$ and 
$\widetilde{T}_{\beta}$.  Therefore we must show there exists an edge between $\widetilde{T}_{\gamma}$ and $\widetilde{T}_{\beta}$ if 
$\col_{\widetilde{T}_{\gamma}}(n) \in \{ c-1,c,c+1\}$.  First note that $\col_{\widetilde{T}_{\gamma}}(n) \not = c+1$ since the cells 
$\{(\ell,c),(\ell,c+1),(\row_{\widetilde{T}_{\gamma}}(n),c+1)\}$ would violate the triple condition~\eqref{equation.triple}.

Assume that $\col_{\widetilde{T}_{\gamma}}(n)=c$.  Then the Dyck pattern $n-2,n,n-1$ appears in $\varphi(\widetilde{T}_{\gamma})$ since $n-1$ 
and $n$ are in the same column and $n$ is therefore in a higher row of the standard Young tableau. The edge from $\varphi(\widetilde{T}_{\gamma})$ to
$I \cdot \varphi(\widetilde{T}_{\gamma})$, where $I=[n-2,n]$, replaces the pattern $n-2,n,n-1$ with $n-1,n,n-2$. The resulting Young composition 
tableau $\widetilde{T}_{\beta}$ has $(n-1)$ in cell $(\ell,c-1)$ while $n$ is in cell $(\ell,c)$ and $n-2$ is in cell $(\row_{\widetilde{T}_{\gamma}}(n),c)$.
Therefore this is the desired edge from $\widetilde{T}_{\gamma}$ to $\widetilde{T}_{\beta}$ in $\QCS_\alpha^+$. 

Next assume that $\col_{\widetilde{T}_{\gamma}}(n)=c-1$.  Then $n,n-2,n-1$ is a Dyck pattern in the reading word of $\varphi(\widetilde{T}_{\gamma})$ 
so that the edge from $\varphi(\widetilde{T}_{\gamma})$ to $I \cdot \varphi(\widetilde{T}_{\gamma})$ with $I=[n-2,n]$ replaces this pattern with 
$n-1,n-2,n$.  Therefore the resulting Young composition tableau is the desired $\widetilde{T}_{\beta}$.  Thus we have identified the desired edge 
from $\widetilde{T}_{\gamma}$ to $\widetilde{T}_{\beta}$ in $\QCS_\alpha^+$ when $\alpha_{\ell} \geqslant 2$.

Finally assume that $\alpha_{\ell}=1$.  Note that $n$ cannot lie in the leftmost column in any row other than row $\ell$ since the leftmost 
column is strictly increasing.  Also, if $n$ were in the second column in a row lower than row $\ell$, the triple condition would be violated.  
Therefore any Young composition tableau $\widetilde{T}_{\gamma}$ of shape $\alpha$ with $n$ not in row $\ell$ must have 
$\col_{\widetilde{T}_{\gamma}}(n) \geqslant 3$.  
But then the edge from $\widetilde{T}_\gamma$ to $t_{n-1}(\widetilde{T}_\gamma)$ is the desired edge.
\end{proof}

\subsection{Contraction to Bruhat order}
\label{section.Bruhat}

Since the (augmented) quasicrystal skeletons tile the crystal skeleton, we thus obtain the crystal analogue of~\eqref{equation.s YQS}:
\begin{equation}
\label{equation.symmetric group}
	\mathsf{CS}(\lambda) = \bigcup_{\alpha \in S_\ell \cdot \lambda} \mathsf{QCS}_\alpha 
	= \bigcup_{\alpha \in S_\ell \cdot \lambda} \mathsf{QCS}^+_\alpha,
\end{equation}
where the union is over all compositions that rearrange $\lambda$. 
Here $S_\ell\cdot \lambda$ indicates the action of the symmetric group $S_\ell$ on $\lambda$, where $\lambda$ has $\ell$ parts. 

Recall from Theorem~\ref{theorem.connectivity} that the augmented quasicrystal skeleton is connected.
Contracting the augmented quasicrystal skeletons inside the crystal skeleton yields \defn{Bruhat order} on $S_\ell$ (or a quotient of Bruhat order if $\lambda$ 
has repeated parts). In particular, under this correspondence $\QCS^+_\alpha$ is associated to $\sigma \in S_\ell$, where $\sigma \cdot \lambda = \alpha$.

It was proven in~\cite[Thm.\ 4.37]{BCDS.2025} that the crystal skeleton $\mathsf{CS}(\lambda)$ restricted to $T\in\mathsf{SYT}(\lambda)$ 
such that $\mathsf{Des}(T)$ has $\ell$ parts, where $\ell$ is the number of parts of $\lambda$, is isomorphic to the $\mathfrak{sl}_\ell$-crystal $B(\lambda)$, 
called the subcrystal property. There is a symmetric group action on crystals (see for example~\cite[\S 2.5]{BumpSchilling.2017} and~\eqref{equation.si}), 
yielding an \defn{extremal crystal}~\cite{Kashiwara.2002}.  
Each quasicrystal skeleton $\QCS_\alpha$ (and equivalently augmented quasicrystal skeleton $\QCS_\alpha^+$) has one canonical element inside 
this $\mathfrak{sl}_\ell$-subcrystal of the crystal skeleton corresponding to the extremal weight vectors. Hence~\eqref{equation.symmetric group} is the 
extension of the extremal crystal structure on the subcrystal of $\mathsf{CS}(\lambda)$ to the full crystal skeleton.

\section{Stanley symmetric functions}
\label{section.stanley}

In this section, we illustrate the results from the preceding section, interpreting them in the context of Stanley symmetric functions.
The Stanley symmetric function $\mathcal{F}_w$ indexed by a permutation $w\in S_n$ was introduced 
in~\cite{Stanley.1984} to enumerate the reduced decompositions of $w$.
Stanley symmetric functions can be defined in terms of \defn{increasing factorizations} of $w^{-1}$.
A word $a^1\cdots a^{\ell}$ is an increasing factorization for $w^{-1}$ into $\ell$ factors if the following is true:
\begin{enumerate}
\item 
We have $a^i = a^i_1 \cdots a^i_{k_i}$ with $a^i_j \in [n-1]$ and $a_1^i<a_2^i<\cdots<a_{k_i}^i$ for all $1\leqslant i\leqslant \ell$. 
Note that $a^i$ can be empty in which case $k_i=0$.
\item
Let $s_{a^i}:=s_{a_1^i}\cdots s_{a_{k_i}^i}$, where $s_j$ for $j\in [n-1]$ are the simple transpositions interchanging $j$ and $j+1$ in $S_n$.
Then 
\[ w^{-1} = s_{a^1} \cdots s_{a^\ell} \textrm{   and    }\ell(w) = \ell(a^1)+\cdots+\ell(a^\ell),\] where $\ell(w)$ is the length of the permutation and 
$\ell(a^i)=k_i$ is the length of the word $a^i$.
\end{enumerate}
Let $\mathcal{IW}^\ell_{w}$ be the set of increasing factorizations of $w^{-1}$ into $\ell$ factors. Then the \defn{Stanley symmetric polynomial} is
\begin{equation}
\label{equation.stanley}
	\mathcal{F}_w(x_1,\ldots,x_\ell) = \sum_{a^1 a^2 \cdots a^\ell \in \mathcal{IW}^\ell_{w}} x_1^{\ell(a^1)} \cdots x_\ell^{\ell(a^\ell)}.
\end{equation}

\subsection{Crystal}
\label{section.crystal Fw}

The crystal for Stanley symmetric polynomials was introduced in~\cite{MS.2016}. We briefly review it here. 

For $a=a^1\cdots a^\ell \in \mathcal{IW}^\ell_{w}$, the \defn{weight} of $a$ is given by $\mathsf{wt}(a) = (\ell(a^1),\ldots,\ell(a^\ell))$. The crystal operator
$f_i$ for $1\leqslant i<\ell$ relies on a \defn{pairing procedure} on the letters in factors $a^i$ and $a^{i+1}$:
\begin{itemize}
    \item Start with the largest letter $x$ in $a^{i+1}$ and pair it with the smallest $y>x$ in $a^i$.
    \item  If there is no such $y$ in $a^i$, then $x$ is unpaired.
    \item The pairing proceeds in decreasing order on elements of $a^{i+1}$, and with each iteration, previously paired letters of $a^i$ are ignored.

\end{itemize}
Let $U_i(a)$ be the set of unpaired letters in $a^i$ after the pairing procedure. 

When $U_i(a)=\emptyset$, we define $f_i(a)=0$. Otherwise $f_i(a)$ is defined by replacing the factors $a^i a^{i+1}$ by
$\tilde{a}^i \tilde{a}^{i+1}$, where
\[
	\tilde{a}^i = a^i \setminus \{t\} \quad \text{and} \quad \tilde{a}^{i+1} = a^{i+1} \cup \{t+s\},
\]
$t = \max U_i(a)$ and $s =\min \{j \geqslant 0 \mid t+j+1 \not \in a^i\}$. 

The raising operators $e_i$ on $a\in  \mathcal{IW}^\ell_{w}$ are 
defined similarly, see~\cite{MS.2016}. We denote the resulting $\mathfrak{sl}_\ell$-crystal by $B_w$.
An example is given in Figure~\ref{figure.B121}.

The crystal structure on reduced words is closely related to 
the \defn{Edelman--Greene bijection}~\cite{EG.1987}, a variant of the Robinson--Schensted--Knuth (RSK) correspondence. Specifically, it is a bijection
between the set $\mathcal{RW}_w$ of reduced words of $w^{-1}$ and tuples of tableaux $(P,Q)$ of the same shape, where the insertion tableau $P$ 
is row and column strict and whose row reading word is in $\mathcal{RW}_w$ and the recording tableau $Q$ is a standard Young tableau. 

The Edelman--Greene bijection is based on the 
following insertion, called \defn{EG insertion}. Inserting the letter $x$ into row $r$ of $P$ is defined by picking out the smallest letter $y>x$ in row $r$:
\begin{itemize}
    \item If no such $y$ exists, the letter $x$ is placed at the end of row $r$.
    \item  If $y=x+1$ and $x$ is also contained in row $r$, then $x+1$ is inserted into row 
$r+1$.
\item  Otherwise, $y$ is replaced by $x$, and $y$ is inserted into row $r+1$. In the last two cases, we say that $y$ has been bumped.

\end{itemize}
From this process, an insertion tableau $P$ and a recording tableau $Q$ are constructed from a reduced word $a_1 \ldots a_d$, starting from 
$P_0=Q_0=\emptyset$ and iteratively defining $P_i$ by inserting $a_i$ into the bottom row of $P_{i-1}$. Letters are bumped until a letter $x$ is 
placed at the end of a row $r$. Then $Q_i$ is defined by adding the letter $i$ at the end of row $r$ in $Q_{i-1}$. Finally, $P=P_d$ and $Q=Q_d$.

The Edelman--Greene bijection can be generalized to a bijection between $\mathcal{IW}_w^\ell$ and pairs of tableaux $(P,Q)$, where $P$ is
as before and $Q$ is a semistandard Young tableau (instead of standard tableau) with letters in $\{1,2,\ldots,\ell\}$, see~\cite{MS.2016}. 
When a letter in the $i$-th factor of $a\in \mathcal{IW}_w^\ell$ is inserted, it is recorded by the letter $i$ in $Q$. Given $a\in \mathcal{IW}^\ell_{w}$, 
define $\mathsf{EG}(a)=(P,Q)$, where $P$ is the insertion tableau and $Q$ is the recording tableau $Q$. We also write $\mathsf{EG}_P(a):=P$ and 
$\mathsf{EG}_Q(a):=Q$. 

It was shown in~\cite{MS.2016} that the crystal operators intertwine with this generalization of the Edelman--Greene bijection:
\begin{theorem} \cite[Theorem 4.11]{MS.2016}
\label{theorem.EG crystal}
Let $a\in \mathcal{IW}_w^\ell$. Then $f_i (\mathsf{EG}_Q(a)) = \mathsf{EG}_Q(f_i a)$.
\end{theorem}

\begin{example}
Let $w=s_3 s_1 s_2 s_3 s_1 s_2$ and take $a=(2)(13)(2)(13) \in \mathcal{IW}^4_{w}$. Then $f_2 a =(2)(3)(12)(13)$,
\[
	\mathsf{EG}(a) = \left(
	\TIKZ[scale=0.35]{
	\Tableau[\scriptsize]{{1,2,3},{2,3},{3}})
	},\ 
	\TIKZ[scale=0.35]{
	\Tableau[\scriptsize]{{1,2,4},{2,3},{4}})
	}\right)
	\quad \text{and} \quad
	\mathsf{EG}(f_2 a) = \left(
	\TIKZ[scale=0.35]{
	\Tableau[\scriptsize]{{1,2,3},{2,3},{3}})
	},\ 
	\TIKZ[scale=0.35]{
	\Tableau[\scriptsize]{{1,2,4},{3,3},{4}})
	}\right).
\]
\end{example}

Define $\mathcal{IW}^{\ell,\mathsf{hw}}_{w} = \{ a \in  \mathcal{IW}^\ell_{w} \mid e_i(a)=0 \text{ for all } 1\leqslant i <\ell\}$. Then as a consequence
of Theorem~\ref{theorem.EG crystal}
\[
	\mathcal{F}_w = \sum_{a\in \mathcal{IW}^{\ell,\mathsf{hw}}_w} s_{\mathsf{wt}(a)}.
\]

While the RSK correspondence is intimately related to the plactic monoid and Knuth relations, the Edelman--Greene correspondence is linked to the 
\defn{Coxeter--Knuth relations} which are, for $x<y<z$, given by
\begin{equation}
\label{equation.CK}
\begin{aligned}
	&\bf{K1.}  &xzy &\equiv_{\text{CK}}  zxy,\\
	&\bf{K2.}  &yxz &\equiv_{\text{CK}} yzx,\\
	&\bf{B.}     &x \; x+1 \; x &\equiv_{\text{CK}} x+1 \; x \; x+1.
\end{aligned}
\end{equation}
Two reduced words $a,b \in \mathcal{RW}_w$ are Coxeter--Knuth equivalent, written $a\equiv_{\text{CK}} b$, if there exists a sequence of Coxeter--Knuth 
relations K1, K2, B transforming $a$ into $b$.
It was shown in~\cite[Theorem 6.24]{EG.1987} that for $a,b \in \mathcal{RW}_w$ we have $a\equiv_{\text{CK}} b$ if and only
if $\mathsf{EG}_P(a) = \mathsf{EG}_P(b)$. 

We give a characterization of the length of the first column of the pair $(P,Q)$ obtained via Edelman-Greene insertion.  This gives an analog of 
the well-known corresponding result for RSK, and will come in handy when studying the quasicrystal and crystal skeleton on reduced words 
in \S\ref{section.quasicrystalreducedwords}.

\begin{proposition}
\label{prop.longest decreasing}
Let $a\in \mathcal{RW}_w$. Then the length of the first column of $\mathsf{EG}_P(a)$ is equal to the length of the longest strictly decreasing
subsequence of $a$.
\end{proposition}

\begin{proof}
The column reading word of $P=\mathsf{EG}_P(a)$ inserts (via Edelman--Greene) to $P$.  It is true for the column reading word of $P$ that the length of its longest
strictly decreasing subsequence is equal to the length of the first column of $P$.
Since $a\equiv b$ if and only if $\mathsf{EG}_P(a) = \mathsf{EG}_P(b)$, the result is hence true for one representative in the equivalence class 
of $a$. Inspecting the relations K1, K2, and B, at most two of the letters in the relations can be part of a strictly decreasing subsequence. Furthermore,
if they are part of such a subsequence on the left, then they still are on on the right and vice versa. This proves the statement.
\end{proof}

\begin{remark}
\label{remark.Greene}
Note that the naive Greene generalization of Proposition~\ref{prop.longest decreasing} to further columns does not hold. For example, take the reduced
word $a=23212$. Its Edelman--Greene insertion gives
\[
\mathsf{EG}(a)=
\left(
	\TIKZ[scale=0.35]{
	\Tableau[\scriptsize]{{1,2},{2,3},{3}})
	},\ 
	\TIKZ[scale=0.35]{
	\Tableau[\scriptsize]{{1,2},{3,5},{4}})
	}\right)
\]
which has column sizes 3 and 2. But a Greene-type theorem would give shape $(3,1,1)$ since there is a longest decreasing subsequence 321, but
then two 2's are remaining.
\end{remark}

\subsection{Quasicrystal}\label{section.quasicrystalreducedwords}

In analogy to Proposition~\ref{proposition.quasi edges}, we characterize crystal edges within a quasicrystal in $B_w$.

\begin{proposition}
Let $a,a'\in \mathcal{IW}^\ell_{w}$ such that $a'=f_i(a)$. Then this is an edge inside a quasicrystal if and only if there is no pairing between
the letters in $a^i$ and $a^{i+1}$.
\end{proposition}

\begin{proof}
This follows from Proposition~\ref{proposition.quasi edges} and the fact that there is a pairing between letters in $a^i$ and $a^{i+1}$ if and only if
there is a pairing between the letters $i$ and $i+1$ in $\mathsf{EG}_Q(a)$ (see~\cite{MS.2016}).
\end{proof}

Recall that $\mathcal{RW}_w$ is the set of reduced words of $w^{-1}$.
Define $\mathsf{red} \colon \mathcal{IW}^\ell_{w} \to \mathcal{RW}_w$ by mapping $a\in \mathcal{IW}^\ell_w$ to the underlying reduced word.
For example, for $a=(2)(13)(2)()(13)$ we have $\mathsf{red}(a)=213213$.

\begin{corollary}
\label{cor.reduced word}
Increasing factorizations $a,b\in \mathcal{IW}^\ell_w$ are in the same quasicrystal if and only if $\mathsf{red}(a)=\mathsf{red}(b)$.
\end{corollary}

\begin{proof}
It suffices to prove the statement for $a,b$ such that $b=f_i(a)$. When there is no pairing between the letters in $a^i$ and $a^{i+1}$, we have
$a_1^i<\cdots <a_{k_i}^i < a_1^{i+1} < \cdots < a_{k_{i+1}}^{i+1}$. In this case $b^i = a_1^i\ldots a_{k_i-1}^i$ and $b^{i+1}
= a_{k_i}^i a_1^{i+1}\ldots a_{k_{i+1}}^{i+1}$ and hence $\mathsf{red}(a)=\mathsf{red}(b)$. If there is a pairing between the letters in $a^i$ and $a^{i+1}$,
we have $a_{k_i}^i > a_1^{i+1}$ and it is not hard to see that the order of the letters in $b$ changes compared to $a$. Hence 
$\mathsf{red}(a)\neq \mathsf{red}(b)$.
\end{proof}

Corollary~\ref{cor.reduced word} implies that the quasicrystals in $B_w$ are indexed by reduced words of $w^{-1}$.
Furthermore, for a given $r\in \mathcal{RW}_w$ the elements of the corresponding quasicrystal are all increasing factorizations into 
$\ell$ factors of $r$. For an example, see Figure~\ref{figure.B121}.

Since the character of a quasicrystal is Gessel's quasisymmetric function indexed by the descent composition associated to the 
quasicrystal, we can rewrite~\eqref{equation.stanley} as
\[
	\mathcal{F}_w = \sum_{a\in \mathcal{RW}_w} F_{\mathsf{Des}(a)}.
\]

\begin{example}
Take $w=s_1 s_2 s_1 \in S_3$. There are two reduced words in $\mathcal{RW}_w$, namely $121$ and $212$ with descent compositions
$(2,1)$ and $(1,2)$, respectively. Hence $\mathcal{F}_{s_1 s_2 s_1} = F_{(2,1)} + F_{(1,2)}$. Compare also with Figure~\ref{figure.B121}.
\end{example}

\subsection{Crystal skeleton}
\label{ss:CS Fw}
Having reviewed the crystal structure on increasing factorizations, we are now interested in the corresponding crystal skeleton.

The crystal skeleton $\mathsf{CS}_w$ associated to the $\mathfrak{sl}_p$-crystal $B_w$ for $p\geqslant \ell(w)$ is obtained by 
contracting the quasicrystals. Hence by Corollary~\ref{cor.reduced word}, the vertices of $\mathsf{CS}_w$ can be labelled by reduced 
words in $\mathcal{RW}_w$. We will now give a more concrete description of the edges in the crystal skeleton.

\begin{definition}
\label{definition.edges CSw}
Let $\ell:=\ell(w)$ be the length of $w\in S_n$. For $a,b \in \mathcal{RW}_w$ with $a=a_1\ldots a_\ell$ and $b=b_1\ldots b_\ell$,
define an edge $a\stackrel{~I~}{\longrightarrow} b$ for $I=[i,i+2m]$ with $m\geqslant 1$ if 
\begin{enumerate}
\item $a_i<\cdots<a_{i+m} > a_{i+m+1} < \cdots < a_{i+2m}$;
\item Setting $a^1=a_i \ldots a_{i+m}$ and $a^2 = a_{i+m+1} \ldots a_{i+2m}$, all letters in $a^2$ are paired under the pairing in 
Section~\ref{section.crystal Fw}. Setting $b^1 b^2 = f_1(a^1 a^2)$, $b$ is obtained from $a$ by replacing the letters $a_i\ldots a_{i+2m}$ with
$\mathsf{red}(b^1 b^2)$.
\end{enumerate}
\end{definition}

\begin{theorem}
The edges in $\mathsf{CS}_w$ are given by the edges in Definition~\ref{definition.edges CSw}.
\end{theorem}

\begin{proof}
Let $a,b\in \mathcal{RW}_w$ be as in Definition~\ref{definition.edges CSw}. Place $a_i\ldots a_{i+m}$ and
$a_{i+m+1}\ldots a_{i+2m}$ into increasing factors and all other letters in $a$ into their own factor to obtain $\tilde{a}\in B_w$. By (2) in 
Definition~\ref{definition.edges CSw} the crystal operator $f_i(\tilde{a})$ is defined and $\mathsf{red}(f_i(\tilde{a}))=b$. Hence each
edge as in Definition~\ref{definition.edges CSw} is indeed an edge in $\mathsf{CS}_w$.

Conversely, suppose that $f_j \tilde{a} = \tilde{b}$ in $B_w$ with $a:=\mathsf{red}(\tilde{a})$, $b:=\mathsf{red}(\tilde{b})$ and $a\neq b$, so that
by Corollary~\ref{cor.reduced word} there is an edge in $\mathsf{CS}_w$. Since $\mathsf{EG}_Q$ intertwines with crystal operators and
on standard tableaux the edges in the crystal skeleton are labeled by intervals $I=[i,i+2m]$ with $m\geqslant 1$, it follows as in the
proof of Corollary~\ref{cor.reduced word} that $a$ must satisfy condition (1) of Definition~\ref{definition.edges CSw}. Under the Edelman--Greene
map, the description of $I$ on the recording tableau is precisely the crystal operator in (2) of Definition~\ref{definition.edges CSw}.
\end{proof}

As in the case of permutations, one can associate a descent composition to reduced words $a\in \mathcal{RW}_w$. Note that a descent in $a$ 
is equivalent to a subword of the form $a_{i-1}.a_{i}$, where $a_{i-1} > a_i$. Let $d(a)$ be the length the descent composition of $a$.
We now consider how the descent composition changes between vertices in $\mathsf{CS}_w$ connected by an edge $I$. 
Recall that by Corollary~\ref{corollary.component change} only length preserving edges in $\mathsf{CS}_w$ can switch components
in the quasicrystal skeleton.

\begin{proposition}
Suppose $a,b \in \mathcal{RW}_w$ are vertices in $\mathsf{CS}_w$ with an edge $I = [i,i+2m]$ from $a$ to $b$. Then 
\begin{enumerate}
    \item $d(a) < d(b)$ if and only if $a_{i+2m} < a_{i+2m+1} < a_{i+m}$ and in the bracketing between $a^1$ and $a^2$, 
    the unique unbracketed element in $a^1$ is $a_{i+m}$.
    \item $d(a) > d(b)$ if and only if $a_i < a_{i-1} < a_{i+1}$ and in the bracketing between $a^1$ and $a^2$, 
    the unique unbracketed element in $a^1$ is $a_{i}$.
\end{enumerate}
Otherwise, $d(a) = d(b)$.
\end{proposition}
\begin{proof}
We prove (1); the argument for (2) is analogous.

If $a_{i+2m} < a_{i+2m+1} < a_{i+m}$ and in the bracketing between $a^i$ and $a^{i+1}$, the unique unbracketed element in $a^{i}$ is $a_{i+m}$, 
then  \[ f_i(a^ia^{i+1}) = a_{i} \cdots a_{i+m-1}.a_{i+m+1}\cdots a_{i+2m} a_{i+m}.\] Since $a_{i+m}> a_{i+2m+1}$, we have created a descent. 
On the other hand, suppose $d(a)< d(b)$. By construction, $f_i(a^ia^{i+1})$ moves the position of a descent, but does not create a new one. 
Hence, the descent must be created away from $a^ia^{i+1}$; note that this can only happen after $a^ia^{i+1}$ since $f_i$ removes an 
element from $a^i$ and adds it to $a^{i+1}$. Since only one term is unbracketed in $a^ia^{i+1},$ we may write $U_i(a) = \{ t \}$. 
Recalling that $s:= \min \{ j \geqslant 0\mid t+j+1 \not \in a^i \},$ we thus have that under the action of $f_i$, the element $t$ is removed from 
$a^i$ and $t+s$ is added to $a^{i+1}$. Since a descent was added from applying $f_i$, it must be that $t+s$ appears at the end of $f_i(a^i a^{i+1})$, 
forcing  $t+s > a_{i+2m+1}$ and $t+s > a_{i+2m}$. Moreover, we must have $a_{i+2m} < a_{i+2m+1}$, or a new descent could not have been created 
by $f_i$.

If $s=0$, since $t$ is the unique unbracketed element in $a^i$ and $a_{i} < a_{i+1} < \cdots < a_{i+m},$ this forces $t = a_{i+m}$, and thus yields the claim.
If $s>0$, then $t+s=a_{i+2m}+1$ since $t+s$ must sit at the end of $f_i(a^i a^{i+1})$. But then $a_{i+2m}<a_{i+2m+1}<a_{i+2m}+1$ which is impossible
for integers. Hence this case cannot happen.
\end{proof}

The example $\mathsf{CS}_w$ for $w=s_1 s_2 s_1 s_3 s_2 s_1$ is given in Figure~\ref{figure.QCS321}.
Similar graphs restricted to the long permutation $w_0$ and edges $|I|=3$ arise from the study of tubing lattices~\cite{DF.2024}.
Note that $\mathcal{F}_{w_0} = s_\delta$, where $\delta$ is the staircase partition of size $\ell(w_0)$. 
Our construction of $\mathsf{CS}_w$ works for all $w\in S_n$.

We can derive the Schur expansion of the Stanley symmetric functions from the crystal skeleton by noting that, when the vertices
of the crystal skeleton indexed by $\lambda$ are labeled by standard Young tableaux, then the superstandard tableau obtained by 
consecutively filling the numbers $1,2,\ldots,|\lambda|$ into $\lambda$ bottom to top, left to right is a natural representative of the 
connected crystal skeleton. These superstandard tableaux correspond to the following reduced words $a\in \mathcal{RW}_w$. Let
$\alpha=\mathsf{Des}(a)$ with $\alpha=(\alpha_1,\ldots,\alpha_d)$ and define the subword
\[
	a^j = a_{\alpha_1+\cdots+\alpha_{j-1}+1}\; a_{\alpha_1+\cdots+\alpha_{j-1}+1} \; \ldots \; a_{\alpha_1+\cdots+\alpha_j} 
	\qquad \text{for $1\leqslant j \leqslant d$.}
\]
Let $\mathcal{RW}_w^{\mathsf{hw}}$ be the set of all $a\in \mathcal{RW}_w$ such that all letters in $a^{j+1}$ are bracketed with the letters 
in $a^j$. In particular, this implies that $\mathsf{Des}(a)$ for $a \in \mathcal{RW}_w^{\mathsf{hw}}$ is a partition.

\begin{corollary}
For $w\in S_n$, we have
\[
	\mathcal{F}_w = \sum_{a\in \mathcal{RW}_w^{\mathsf{hw}}} s_{\mathsf{Des}(a)}.
\]
\end{corollary}

\begin{example}
Let $w=s_3 s_5 s_1 s_2 s_4 s_3 s_1$. Then the elements in $a \in \mathcal{RW}_w$ such that $\mathsf{Des}(a)$ is a partition are
\begin{align*}
	& 13.25.14.3 && 13.25.4.3.1 && 35.14.23.1 && 15.34.23.1\\
	& 35.12.14.3 && 135.24.3.1 && 135.24.13 &&
\end{align*}
Of these, only $135.24.13$ and $135.24.3.1$ have the property that all letters in $a^{j+1}$ are paired with letters in $a^j$. Hence
$\mathcal{F}_w = s_{(3,2,2)} + s_{(3,2,1,1)}$.
\end{example}

\subsection{Quasicrystal skeleton}
Finally, we consider the quasicrystal skeleton on reduced words.

To determine whether the edge $a \stackrel{I}{\longrightarrow} b$ in $\mathsf{CS}_w$ with $I=[i,i+2m]$ changes the quasicrystal skeleton component,
it suffices to check whether $\mathsf{EG}_Q(a)$ satisfies the conditions in Theorem~\ref{theorem.component change}. 

Some of the conditions in Theorem~\ref{theorem.component change} can be checked directly on $a \in \mathcal{RW}_w$. For example, to check that 
$i$ appears in the first column of $\mathsf{EQ}_Q(a)$, by Proposition~\ref{prop.longest decreasing} one just needs to check that the length of the longest strictly decreasing subsequence in
$a_1\ldots a_{i-1}$ increases by one compared to $a_1\ldots a_i$. Similarly, to check that $k>i$, by Proposition~\ref{prop.longest decreasing} it suffices
to check that the length of the longest decreasing subsequence in $a_1\ldots a_{i+m}$ increases by one  compared to $a_1 \ldots a_{i+m+1}$. By Remark~\ref{remark.Greene} it is not possible to do the same for subsequent columns.

However, Reiner and Shimozono~\cite{RS.1995} defined the \defn{plactification} of a reduced word of $w^{-1} \in S_n$
\[
	\mathfrak{p} \colon \mathcal{RW}_w \to [n-1]^{\ell(w)}
\]
such that for $a,a' \in \mathcal{RW}_w$ with $\mathsf{EG}_P(a) = \mathsf{EG}_P(a')$ we have $P(\mathfrak{p}(a)) = P(\mathfrak{p}(a'))$,
where $P(v)$ denotes the RSK insertion tableau of the word $v$. Furthermore $\mathsf{EG}_Q(a) = Q(\mathfrak{p}(a))$, where $Q(v)$ denotes
the RSK recording tableau of $v$. The plactification map $\mathfrak{p}$ on $a=a_1\ldots a_\ell \in \mathcal{RW}_w$ is defined in terms of crystal reflection 
operators $s_i$ on words as defined in~\eqref{equation.si} as follows:
\[
	\mathfrak{p}(a) := a_1 s_{a_1}\left( a_2 s_{a_2} \left( \cdots a_{\ell-1} s_{a_{\ell-1}}(a_\ell) \cdots \right) \right). 
\]

We may now give a way to check \ref{condition1} of Theorem~\ref{theorem.component change} without reference to the Edelman--Greene insertion,
but rather using \defn{Greene's Theorem}~\cite{Greene.1974} (see also~\cite[Theorem 3.5.3]{Sagan.2001}). Suppose $v$ is a word, $\lambda$ is the shape
of $P(v)$, and $\lambda^t$ is the transpose of $\lambda$. Then Greene's Theorem states that
\[
	\lambda^t_1+ \lambda^t_2 + \cdots + \lambda^t_j = \text{length of a longest $j$-decreasing subsequence of $v$.}
\]
Here a \defn{$j$-decreasing subsequence} of $v$ is a union of $j$ mutually distinct decreasing subsequences $v^{(1)},\ldots, v^{(j)}$ of $v$.

\begin{theorem}
\label{theorem.EG conditions}
Suppose $a \stackrel{I}{\longrightarrow} a'$ is an edge in $\mathsf{CS}_w$ with $I=[i,i+2m]$. Then $\mathsf{EG}_Q(a)$ satisfies \ref{condition1} of 
Theorem~\ref{theorem.component change} if and only if in $p=p_1\ldots p_\ell = \mathfrak{p}(a)$ 
\begin{enumerate}
\item the length of a longest $j$-decreasing subsequence of $p_1\ldots p_{i-1+j}$ is $j$ larger than a longest $j$-decreasing subsequence of 
$p_1\ldots p_{i-1}$ for $1\leqslant j \leqslant k-i+1$ for some $k>i$, and
\item the lengths of the longest $(k-i+1)$-decreasing subsequence in $p_1\ldots p_{k+m}$ and $p_1 \ldots p_{i+2m}$ are the same.
\end{enumerate}
\end{theorem}

\begin{proof}
This follows directly from Remark~\ref{YCTrows}, Greene's Theorem and the properties of the plactification map~$\mathfrak{p}$.
\end{proof}

\begin{example}
Consider the reduced word $a=23212$ from Remark~\ref{remark.Greene} and $I=[1,3]$. In Remark~\ref{remark.Greene}, the Edelman--Greene
insertion and recording tableaux were computed. Note that $\mathsf{EG}_Q(a)$ satisfies \ref{condition1} of Theorem~\ref{theorem.component change}
with $k=2$. We compute the plactification as 
\[
	\mathfrak{p}(a) = 2 s_2 \left( 3 s_3 \left( 2 s_2 \left( 1 s_1 2 \right) \right) \right) = 23211.
\]
The longest 1-decreasing subsequence of $\mathfrak{p}(a)$ is $321$ which is of length 3 and the longest 2-decreasing subsequence is 321,21 which is 
of length 5 confirming that the column lengths of $\mathsf{EG}_P(a)$ are indeed 3 and 2.

Furthermore, the longest $1$-decreasing (resp. $2$-decreasing) subsequence of $2$ (resp. $23$) is one (resp. two) longer than the longest $j$-decreasing
subsequence of $\emptyset$ for $1\leqslant j \leqslant 2$, which are of length zero. This confirms (1) in Theorem~\ref{theorem.EG conditions}.
Condition (2) is trivially satisfied since $k+m=i+2m=3$.
\end{example}

\ref{condition2} of Theorem~\ref{theorem.component change} is much harder to check directly on the reduced words $a \in \mathcal{RW}_w$
due to the fact that the bijection $\varphi^{-1}$ of Definition~\ref{def.phi} is involved in computing the composition shape $\alpha$.

\bibliographystyle{plain}
\bibliography{paper}
\end{document}